\documentclass[a4paper;11pt]{amsart}
\usepackage{mathrsfs}
 \usepackage[all]{xy}
\usepackage{amsmath,amssymb,amscd,bbm,amsthm,mathrsfs}
\newtheorem{thm}{Theorem}[section]
\newtheorem{lem}{Lemma}[section]
\newtheorem{cor}{Corollary}[section]
\newtheorem{prop}{Proposition}[section]
\newtheorem{rem}{Remark}[section]

\theoremstyle{definition}

\begin{document}
\numberwithin{equation}{section}

\title[ Quaternionic projective invariance of  the $k$-Cauchy-Fueter complex and applications I.]
 {   Quaternionic projective invariance  of  the $k$-Cauchy-Fueter complex  and applications   I. }
\author{  Wei Wang}
\thanks{
Supported by National Nature Science Foundation in China (No.
11971425)  }\thanks{   Department of Mathematics,
Zhejiang University, Zhejiang 310027,
 P. R. China, Email:   wwang@zju.edu.cn}

\begin{abstract}   The $k$-Cauchy-Fueter complex in quaternionic analysis  is   the counterpart of the Dolbeault complex in  complex analysis.
In this paper, we find the   explicit   transformation formula of   these complexes under
${\rm SL}(n+1,\mathbb{H})$, which acts on $\mathbb{H}^{ n}$ as  quaternionic   fractional linear transformations. These transformation formulae have several
interesting applications to  $k$-regular functions, the quaternionic counterpart of holomorphic functions, and geometry of domains. They allow  us to construct
  the $k$-Cauchy-Fueter complex over locally   projective flat manifolds explicitly
and introduce various notions of pluripotential theory on this kind of manifolds. We also introduce a quaternionic projectively invariant operator from
the quaternionic Monge-Amp\`{e}re operator, which can be used to find projectively invariant defining density
of a domain,  generalizing
Fefferman's construction in  complex analysis.   \end{abstract}

\maketitle
\section{Introduction}
Since 1980s, people have been interested in developing analysis of several quaternionic  variables \cite{P88}. The quaternionic counterpart of the Cauchy-Riemann
operator
is a family of operators   acting on  $ \odot^{k} \mathbb{C}^2 $-valued functions, called the $k$-Cauchy-Fueter operator,
$k=0,1,\ldots$,  because the group ${\rm SU}(2)$  of unit quaternionic numbers has a family of  irreducible representations $ \odot^{k} \mathbb{C}^2 $,   while the
group of unit complex numbers   has only one irreducible representation space $\mathbb{C}$. As the quaternionic counterpart of the  Dolbeault complex,
    the $k$-Cauchy-Fueter  complexes  on the flat space $\mathbb{H}^n$ are  known explicitly now (cf. \cite{bS} \cite{bures} \cite{CSSS} \cite{CSS} \cite{Wa10}
and references therein):
\begin{equation}\label{eq:CF-complex}
0\rightarrow \Gamma(D,  \mathcal{ {V}}_0 )
 \xrightarrow{ \mathcal{ {D}}_{0} } \Gamma(D,  \mathcal{ {V}}_1   )\xrightarrow{  \mathcal{ {D}}_{1} }  \cdots \xrightarrow{  \mathcal{{D}}_{2n-2}}
 \Gamma(D,  \mathcal{{V}}_{2n-1} )\rightarrow0,
 \end{equation}for a domain $D$ in $\mathbb{H}^n$,
where  $\Gamma(D,  \mathcal{ {V}}_j )$ is the space of $ \mathcal{{V}}_j $-valued smooth functions, and
\begin{equation*}
   \mathcal{{V}}_j : =\left\{ \begin{array}{ll}  \odot^{k-j}\mathbb{C}^{2 }\otimes
 \wedge^j\mathbb{C}^{2n*} ,\qquad \qquad &j=0,\cdots, k ,\\
  \odot^{j -k-1}\mathbb{C}^{2*}\otimes\wedge^{j+1}\mathbb{C}^{2n*},\qquad &j=k+1,\cdots, 2n-1.
   \end{array}\right.
 \end{equation*} They have several interesting applications to quaternionic analysis, e.g. to the quaternionic Monge-Amp\`{e}re operator and  quaternionic
 plurisubharmonic functions (cf. \cite{wan-wang,wang21} and references therein).
A function   $f\in \Gamma( D,  \odot^k \mathbb{C}^2 )$ is   called {\it $k$-regular} if
$
   \mathcal{ D}_0f=0
$ on $D$.
The space of all $k$-regular functions on  $D $ is denoted by $\mathcal{O}_k(D)$. Because of Hartogs'  phenomenon  for
  $k$-regular functions (see e.g. \cite{CSSS,P88,Wa10}), it is a fundamental problem to characterize domains of $k$-regularity, the  quaternionic counterpart of
  domains of holomorphy, and domains with vanishing cohomology of   the $k$-Cauchy-Fueter complex.

It is a useful and   important fact in complex analysis that the product of two  holomorphic functions is also holomorphic, and so is the composition of two
holomorphic transformations.
Moreover, the Cauchy-Riemann operator and   Dolbeault complex are invariant under   biholomorphic transformations,
and so they exist on complex manifolds.
 The counterpart of a holomorphic transformation  is the notion of a regular transformation $f:\mathbb{H}^{ n}\rightarrow \mathbb{H}^{ n}$, i.e. each component of
 $f$ is a
 $1$-regular $  \mathbb{H} $-valued function. But the composition of a $k$-regular function with a regular transformation is usually
 not $k$-regular because of non-commutativity. It is necessary to known under which transformations of $ \mathbb{H}^{ n} $ the $k$-regularity is preserved, i.e,  to clarify the
 invariant group of  the $k$-Cauchy-Fueter operator and    complex.
Liu-Zhang  \cite{Zhang} constructed and investigated invariant operators on the    quaternionic hyperbolic space under the action of ${\rm Sp}(n ,1) $ by using
representation theory, which coincide with the $k$-Cauchy-Fueter operator for $k\geq 1$.

The  Cauchy-Riemann operator  is unique in the sense that
it has an invariant group of infinite dimensions, while for all known generalizations, such as the Dirac operator    in
Clifford analysis and the tangential Cauchy-Riemann operator etc., their invariant groups are only  of finite dimensions. But they are still large enough
to have various
important applications (cf. e.g. \cite{PQ,Ryan95}).  More generally, it is an active direction to investigate the function theory of  conformally invariant
operators
of   higher spins   (cf. \cite{BEV,BDLSW,DWR17,ER,ES,LSW}  and references therein). On $4$-dimensional Minkowski
space, they are massless field operators for higher spins in physics, which are systematically investigated by Penrose et. al. \cite{PR1,PR2} with the help of
conformal
invariance. Operators of higher spins on the Euclidean space and  on the Minkowski space have the same complexification.
They are also explored from the point of view of representation theory by Frenkel-Libine
\cite{FL1,FL2}.

The $1$-Cauchy-Fueter complex has been studied by using
  commutative algebra and  computer algebra method   since  90s (cf. \cite{CSSS} and references therein).  Meanwhile,    Baston   \cite{Ba} constructed a
  family of   quaternionic  complexes over
 complexified
quaternionic-K\"ahler manifolds  by using  the twistor method, generalizing  Eastwood-Penrose-Wells result  for $n=1$ \cite{Eastwood}.
The twistor construction implies the invariance of complexes   under the action of ${\rm SL}(2n+2,\mathbb{C})$, which was used to find explicit form of operators
in the  complexified version of   complexes \cite{bS} \cite{bures}   \cite{CSS}. See also \cite{CC,CaSS,CS,Salac} and references therein for the construction of
invariant differential operators and complexes. Several interesting
differential complexes over curved  manifolds have been constructed from   BGG sequences \cite{CSS01} \cite{CS2} associated to a   semisimple Lie algebra $\mathfrak g$ and a parabolic subalgebra. This construction can be applied to quaternionic manifolds. But the kernel  of the first operator of a BGG sequence  is a    finite dimensional irreducible representation  of $\mathfrak g$, while for the $k$-Cauchy-Fueter complex, the kernel  of  the first operator (i.e. the $k$-Cauchy-Fueter operator) is  of infinite dimensional. So it is not a
  BGG sequence.
 In this paper, we find the transformation formula of each operator  $ \mathcal{ D}_j $ in \eqref{eq:CF-complex} under the action of
${\rm SL}(n+1,\mathbb{H})$, which acts on $\mathbb{H}^{ n}$ as  quaternionic   fractional linear transformations. This transformation formulae have several
important applications to  $k$-regular functions  and geometry of domains.

Let $  \mathbb{C}^{2 }$  be the standard   $\text {GL}(1,\mathbb{H}
  )$-module and let $  \mathbb{C}^{2n }$
  be   the standard ${\rm {GL}}(n,\mathbb{H})$-module.
Let $\mathbb{C}^{ 2*}$ and  $\mathbb{C}^{ 2n *}$ be   modules dual to  $\mathbb{C}^{ 2}$ and  $\mathbb{C}^{ 2n }$ , respectively. They are trivially extended to be
  \begin{equation*}
     G_0  =S (\text {GL}(1,\mathbb{H}
  )\times {\rm {GL}}(n,\mathbb{H}))=(\text {GL}(1,\mathbb{H}
  )\times {\rm {GL}}(n,\mathbb{H}))\cap {\rm
 {SL}}(n+1,\mathbb{H})
  \end{equation*}
  modules.
It is convenient to identify a $  \odot^{\sigma}\mathbb{C}^{2}
\otimes  \wedge^\tau\mathbb{C}^{2n*}$-valued function $f$ with     a function in variables $\mathbf{q}\in \mathbb{H}^n$,   $s_{ {A}'}\in \mathbb{C}^2$  and
Grassmannian variables
$\omega^A$, which is homogeneous of degree $\sigma$ in $s_{A'}$ and  of degree $\tau$ in $\omega^{  A}$, i.e.
\begin{equation} \label{eq:f-supervariables}
  f= f^{\mathbf{A}'}_{\mathbf{A}}(\mathbf{q})s_{\mathbf{A}'}\omega^{\mathbf{A} }
\end{equation}
where $s_{\mathbf{A}'}:=s_{A_1'} \ldots  s_{A_\sigma'}$ for the multi-index $\mathbf{A}' = A_1' \cdots A_\sigma'$, and  $\omega^{ \mathbf{A}  } :=\omega^{ A_1 }
\cdots \omega^{A_\tau} $ for the multi-index
$\mathbf{A} = A_1 \cdots  A_\tau$  ($ {A}_j=0,\ldots, 2n-1$, $A_l'=0',1'$). Here and in the sequel, We   use the Einstein convention of taking summation for
repeated indices.

Write
  an element $g  \in{\rm
 {SL}}(n+1,\mathbb{H}) $ as
\begin{equation}\label{eq:g-H}
g^{-1}=\left(  \begin{array}{cc}\mathbf   {a}_{1\times 1}&\mathbf   {b}_{1\times n}\\
 \mathbf  {c}_{n\times 1}& \mathbf  {d}_{n\times n}
   \end{array}\right)
\end{equation}
where $ \mathbf  {a},\mathbf   {b},
\mathbf   {c}$ and $\mathbf   {d}$ are  quaternionic    matrices. It defines a fractional linear transformation of $\mathbb{H}^n$:
\begin{equation} \label{eq:g.q}
  T_{g^{-1}}:\mathbf{z}\rightarrow  g^{-1}.\mathbf   {q} :=(\mathbf   {c}+ \mathbf  {d}\mathbf   {q} ) ( \mathbf  {a}+\mathbf   {b}\mathbf   {q})^{-1} ,
\end{equation}
Denote
\begin{equation*}\begin{split}
  J_1 (g^{-1} ,\mathbf{q} )&:=   \mathbf{a}+\mathbf{b}\mathbf{q} \in  \mathbb{H} ,\\
   J_2(g^{-1} ,\mathbf{q} )&:=  \mathbf{d} -(\mathbf{c}+\mathbf{d}\mathbf{q} ) (\mathbf{a}+\mathbf{b}\mathbf{q})^{-1}\mathbf{b} \in {\rm {GL}}(n,\mathbb{H}).
 \end{split} \end{equation*}

If $j\leq k $, an element $g$ in \eqref{eq:g-H} acts on  $f\in \Gamma(\mathbb{H}^{ n}, \mathcal  {V}_j )$ given by \eqref{eq:f-supervariables} as
 \begin{equation}\label{eq:Uj<}\begin{split}
       [\pi_j(g)f] (\mathbf{q}):=&\frac {f_{\mathbf{A}}^{\mathbf{A}'}(g^{-1}.\mathbf{q}) }{|\mathbf a+\mathbf b\mathbf q|^{2(j+1)}}  J_1 (g^{-1} ,\mathbf{q} ) ^{-1}
       .
       s_{\mathbf{A}'}J_2 (g^{-1} ,\mathbf{q} ) . \omega^{\mathbf{A} }.
   \end{split} \end{equation}
  If $j\geq k $,
 we write $ f\in \Gamma(\mathbb{H}^{ n}, \mathcal  {V}_j ) $ as
\begin{equation}\label{eq:f-supervariables>}
  f= f_{\mathbf{A} \mathbf{A}'} s^{\mathbf{A}'} \omega^{\mathbf{A} },
\end{equation}  where we use
   $ s^{ {A}'} $ as coordinate functions on $\mathbb{C}^{2 *}$.
An element $g$ in \eqref{eq:g-H}  acts  as
\begin{equation}\label{eq:Uj>}\begin{split}
       [\pi_j(g)f] (\mathbf q):=&\frac {f_{\mathbf{A}\mathbf{A}'}(g^{-1}.\mathbf q) }{|\mathbf a+\mathbf b\mathbf  q|^{2(j+1)}}  J_1 (g^{-1} ,\mathbf{q}
       )    . s^{\mathbf{A}'}J_2 (g^{-1} ,\mathbf{q} )  . \omega^{\mathbf{A} } .
   \end{split} \end{equation}

   $\pi_j$ is not  a real representation on $  \Gamma(\mathbb{H}^{ n}, \mathcal  {V}_j )$, because for $f\in \Gamma(\mathbb{H}^{ n}, \mathcal  {V}_j )$,
   $\pi_j(g)f$
   is
   singular on the quaternionic hyperplane
   \begin{equation*}
  \mathcal{L}_g:=\left\{ q\in\mathbb{H}^n;   \mathbf{a}+\mathbf{b q}=0\right\}.
   \end{equation*}
  But outside of singularities, it still satisfies the identity of a  representation:
  \begin{equation}\label{eq:representation}
     \pi_j(g_1)\pi_j(g_2)f=\pi_j(g_1g_2)f  .
  \end{equation}

\begin{thm}\label{thm:k-invariant}
   $\mathcal  {D}_j  $ is   ${\rm SL}(n+1,\mathbb{H})$-invariant, i.e.
\begin{equation*}\mathcal  {D}_j (\pi_j(g)f)=
    \pi_{j+1}(g) \mathcal  {D}_j f.
\end{equation*}
for any $g\in {\rm SL}(n+1,\mathbb{H})$ and $f\in \Gamma(\mathbb{H}^{ n}, \mathcal  {V}_j )$.
\end{thm}

The invariance implies that if  $f$ is $k$-regular on a domain $D\subset\mathbb{H}^n$, then
 \begin{equation}\label{eq:k-regular-g}
   \frac {1}{|\mathbf{a}+\mathbf{bq}|^{2 }}  (\mathbf{a}+\mathbf{b q})^{-1}   .  f((\mathbf{c}+\mathbf{d q})(\mathbf{a}+\mathbf{b q})^{-1})
 \end{equation}
 is also $k$-regular  on  $g .D \setminus\mathcal{L}_{g }$  for any $
   g \in {\rm
 {SL}}(n+1,\mathbb{H})
$.
 In particular, if we take a $\mathcal  {V}_0 $-valued constant function $f=s_{\mathbf{A}'}$, then the rational function
\begin{equation}\label{eq:k-regular-fraction}
   \frac {1}{|\mathbf{a}+\mathbf{bq}|^{2 }}  (\mathbf{a}+\mathbf{b q})^{-1}   . s_{\mathbf{A}'}
 \end{equation}
 is  $k$-regular. This allows us to introduce the quaternionic version of the
Fantappi\`e transformation
 \eqref{eq:Fantappie }, and leads to
   an interesting question when any $k$-regular function on a subset of $\mathbb{H}^n$  is the
 superposition of the simple rational functions of the form \eqref{eq:k-regular-fraction}.

 A domain $D\subset \mathbb{H}^n$ is called {\it  (quaternionic)  linearly convex} if for any $\mathbf{p}\in\partial D$, there is an hyperplane of quaternionic
dimension $n-1$ passing through $\mathbf{p}$ and not intersecting $D$.  This notion is the generalization of
the complex one. As a consequence, a linearly convex domain is a   domain  of $k$-regularity.

 A manifold  is called {\it locally (quaternionic) projective flat} if    it has coordinates
charts $\{(U_{\alpha},\phi_{\alpha})\}$ with $\phi_{\alpha}:U_{\alpha}\rightarrow \mathbb{H}^n$  and transition  maps
\begin{equation}\label{eq:charts}
   \phi_{\beta}\circ \phi_{\alpha}^{-1}: \phi_{\alpha}(U_{\alpha}\cap U_{\beta})\longrightarrow \phi_{\beta}(U_{\alpha}\cap U_{\beta})
\end{equation}
  given by the induced action \eqref{eq:g.q} for some  $g \in \text{SL} (n+1,\mathbb{H}).$ If $\Gamma$ is a discrete subgroup of $ \text{SL} (n+1,\mathbb{H})$,
  $\mathbb{H}P^{n }/\Gamma$
is a locally  projective flat manifold.
The  quaternionic hyperbolic space can be realized as  the unit ball $  B^{4n } $, whose  group of  isometric
automorphisms is ${\rm Sp}(n ,1) \subset \text{SL} (n+1,\mathbb{H})$. If $\Gamma$ is a discrete subgroup of ${\rm Sp}(n ,1)$, then $  B^{4n
}/\Gamma$ is a locally   projective flat manifold.  In particular,  if $\Gamma$  is a cocompact or   convex cocompact
 subgroup of ${\rm Sp}(n ,1) $,
then $  B^{4n }/\Gamma$ is a compact locally   projective flat manifold without   or with  boundary (a
 spherical  quaternionic contact manifold) \cite{Shi-Wang}.

 $J_\mu$ is a cocycle, i.e.
\begin{equation}\label{eq:pi-q}
    J_\mu( g_2^{-1}  g_1^{-1}   ,\mathbf{q} ) =J_\mu(g_2^{-1} ,g_1^{-1} .\mathbf{q}
    ) J_\mu(g_1^{-1} ,\mathbf{q}),\qquad  \mu=1,2.
\end{equation}
$ J_1 ^{-1} $ can be used to glue trivial $  \mathbb{C}^{2  }$-bundles to obtain the bundle $H $. We  use $J_1  $  to glue trivial $  \mathbb{C}^{2 * }$-bundles to obtain the bundle $H^*  $. Here the action $ J_1 $ on the
representation  $  \mathbb{ C}^{2 * }$ is dual to the action of $J_1 ^{-1} $.  While $J_2 $ can be used to glue trivial $
\mathbb{C}^{2n*}$-bundles
to obtain the bundle $E^* $. Let
$\wedge^{\tau} {E}^*$ be the $\tau$-th exterior product of $ {E}^*$, and let $\odot^{\sigma}{H}$ and  $\odot^{\sigma}{H}^* $  be the $\sigma$-th symmetric products
of $ H$ and  $ {H}^* $, respectively.
There also exists a distinguished line bundle $\mathbb{R}[-1]$ so that
\begin{equation*}
   \wedge^{4n} T^*M\cong\mathbb{R}[-2n-2],
\end{equation*}
where $\mathbb{R}[-l]=\otimes^l\mathbb{R}[-1]$, and $\mathbb{C}[-1]\cong \wedge^2{H}^*$. Denote
$
  V[-l]:=V\otimes \mathbb{R}[-l] 
$ for a vector bundle $V$.
On a  locally   projective flat manifold $M$, we have the the $k$-Cauchy-Fueter complex:
 \begin{equation}\label{eq:quaternionic-complex-diff}
0\rightarrow \Gamma(M,  \mathcal{ V }_0 )
 \xrightarrow{ \mathcal{ {D}}_{0} } \Gamma(M,   \mathcal{ V }_1   )\xrightarrow{  \mathcal{ {D}}_{1} }  \cdots \xrightarrow{  \mathcal{{D}}_{2n-2}}
 \Gamma(M,   \mathcal{ V }_{2n-1} )\rightarrow0,
 \end{equation}
where
\begin{equation*}
  \mathcal{ V }_j : =\left\{ \begin{array}{ll}   \odot^{k-j
}{H}\otimes \wedge^j {E}^*  [-j-1]  ,\qquad \qquad &j=0,\cdots, k ,\\
 \odot^{j -k-1}{H}^* \otimes \wedge^{ j+1}
{E}^*  [ -j-1]  ,\qquad &j=k+1,\cdots, 2n-1.
   \end{array}\right.
 \end{equation*}
$k=0,1,\ldots$. For $k=0$, $\mathcal  {D}_0 $ is
the {\it  Baston operator}
$
\triangle:  \Gamma\left(  M,\mathbb{R} [ -1]\right)
\longrightarrow\Gamma \left(   M, \wedge^2 {E}^* [ -2] \right).
$ A    upper
semicontinuous  section of $  \mathbb{R} [-1]$
   is said to be  {\it plurisubharmonic}  if    $\triangle u$ is
a closed positive $2$-current. The quaternionic Monge-Amp\`{e}re operator on a  locally   projective flat manifold is defined as $(\triangle u)^n: \Gamma\left( M,
\mathbb{R} [ -1]\right)
\longrightarrow\Gamma \left(  M,  \wedge^{2n} {E}^* [ -2n] \right)$.

Recall that a {\it quaternionic-K\"ahler manifold} $M$  is  a Riemannian manifold whose Levi-Civita connection    preserves
the   quaternionic structure, i.e.    the frame
bundle of $M$   reduces   to    a  principal  ${\rm {Sp}}(n ){\rm {Sp}}(1)$-bundle with a torsion-free connection.  The quaternionic hyperbolic  space $  B^{4n } $  is quaternionic-K\"ahler, and so is  $
B^{4n }/\Gamma$ for a discrete subgroup  $\Gamma$ of the isometric
  group ${\rm Sp}(n ,1)$ of the  quaternionic hyperbolic metric. But for   a discrete subgroup $\Gamma$  of $ \text{SL}
 (n+1,\mathbb{H})$,  the locally   projective flat manifold    $\mathbb{H}P^{n }/\Gamma$  is not quaternionic-K\"ahler in general, since the manifold may have nonvanishing torsion.
   The construction of locally   projective flat manifolds is easy, because   we don't need to construct   special    connections on
them.

 Alesker \cite{alesker2} constructed and investigated the quaternionic Monge-Amp\`{e}re operator on quaternionic-K\"ahler manifolds by    the twistor method and
 method of
   complexification  of such manifolds by Baston \cite{Eastwood}.   The quaternionic Monge-Amp\`{e}re operator in \cite{alesker2} is defined in terms of the
   quaternionic-K\"ahler connection, while
   on locally   projective flat manifolds, the quaternionic Monge-Amp\`{e}re operator is easily defined. Moreover, it allows us to introduce various notions of
   pluripotential theory on this kind of manifolds, in particular,
closed positive  currents and their ``integrals", etc.

The paper is organized as follows. In Section 2, we describe the complexified version of  the $k$-Cauchy-Fueter complex over the complex space $\mathbb{C}^{2n\times 2}$, on which $ {\rm SL}(2n+2,
\mathbb{C})$
acts   as complex fractional linear transformations. In Section 3, the $ {\rm SL}(2n+2, \mathbb{C})$-invariance of   the complexified  $k$-Cauchy-Fueter
complex is proved.
It is reduced to its $\mathfrak  {sl}(2n+2, \mathbb{C})$-invariance, which can be checked more easily and directly.  In Section 4,   the  ${\rm
SL}(n+1,\mathbb{H})$-invariance
in
Theorem \ref{thm:k-invariant} is deduced from  the  $ {\rm SL}(2n+2, \mathbb{C})$-invariance by using the embedding of  $\mathbb{H}^{n }$ to $\mathbb{C}^{2n\times
2}$.
Cocycles
$J_\mu$'s are used to construct the bundles $H$, $H^*$ and $E^*$ over locally   projective flat manifolds and  the $k$-Cauchy-Fueter complex exists over such
manifolds.
In Section 5,  we   introduce various notions of pluripotential theory on locally   projective flat manifolds.  In Section 6, we construct
a quaternionic projectively invariant operator from the quaternionic Monge-Amp\`{e}re operator, which can be used to find projectively invariant defining density
of a domain,  as
Fefferman \cite{Feff} did  in the complex case  and Sasaki \cite{Sasaki85} and Marugame  \cite{Marugame16} \cite{Marugame18} did     for     locally real
projective flat manifolds. This defining density
will be used to constructed various projectively invariants   as Fefferman constructed CR invariants  of  boundaries    and
CR invariant differential operators  on  boundaries  in the subsequent part. It is also interesting to consider the quaternionic version of the generalization of
Fefferman-type constructions to curved projective manifolds \cite{CGH}.
\section{The complexified version of  the $k$-Cauchy-Fueter complex}
 \subsection{  $ {\rm SL}( n+1,
\mathbb{H})$ and its complexification  }

 The Lie algebra of $G=  {\rm SL} (n + 1, \mathbb{H})$ is $\mathfrak{g}  =\mathfrak{sl} (n + 1, \mathbb{H})=\{A\in \mathfrak{gl} (n + 1, \mathbb{H}); {\rm Re}\,{\rm Tr} A=0\}$. Let $
\mathfrak{g}_{0}=\mathfrak{s }(\mathfrak{gl} (  1, \mathbb{H}) \oplus \mathfrak{gl} (n  , \mathbb{H}))=\mathfrak{sl} (  1, \mathbb{H}) \oplus \mathfrak{sl} (n  ,
\mathbb{H})\oplus \mathbb{R}$. $\mathfrak{g} = \mathfrak{g}_{-1}\oplus \mathfrak{g}_{0}\oplus\mathfrak{g}_{1}$ with the grading  easily visible in a block form with blocks
of sizes $1, n$:
\begin{equation*}
   \mathfrak{g}_{-1}=\left\{ \left(  \begin{array}{cc} 0& 0\\
*&0
   \end{array}\right)\right\},\qquad  \mathfrak{g}_{0}=\left\{ \left(  \begin{array}{cc} *& 0\\
0&*
   \end{array}\right)\right\},\qquad  \mathfrak{g}_{ 1}=\left\{ \left(  \begin{array}{cc} 0& *\\
0&0
   \end{array}\right)\right\}.
\end{equation*} Thus, $ \mathfrak{g}_{- 1}\cong\mathbb{H}^n\cong \mathfrak{g}_{  1}$.
Its Lie brackets are given by  $[X,Y]=XY-YX$ for $X, Y\in \mathfrak{g} $.

 Write an element $g$ of ${\rm G} ^{\mathbb{C}}= {\rm SL}(2n+2,
\mathbb{C})$ as $(g_\beta^\alpha)$ with $\alpha,\beta=0',1',0,1,\cdots,2n-1$.
We adopt the following index notations:
 $ A,B,C,  \cdots \in\{ 0,1,\cdots,2n-1\},$ $
  A',B',C', \cdots \in\{0',1' \},
 $. Then we can write
\begin{equation}\label{eq:matix}
  g =  \left(    g_{\beta}^{\alpha} \right)=  \left(  \begin{array}{cc} g_{B'}^{A'}& g^{A'}_B\\
g_{B'}^A&g_B^A
   \end{array}\right),
\end{equation}
where lower indices are column ones, while upper indices are row ones. It acts on vector $\left(
    \begin{array}{c}u^{A' } \\ u^{A}
       \end{array}
\right)\in \mathbb{C}^{2(n+1)}$. As a dual module, element $\left(
u_{A' }\quad
      u_{A}
\right)$ in $ \mathbb{C}^{2(n+1)*}$ is acted by matrix \eqref{eq:matix} from right, i.e. $g. v_\alpha=\left(
u_{A' }\quad
      u_{A}
\right)g^{-1} $.

Denote by $\mathfrak{g}_{-1}^{\mathbb{C}}$, $\mathfrak{g}_{0}^{\mathbb{C}}$ and $\mathfrak{g}_{ 1}^{\mathbb{C}}$   subalgebras of the following forms
\begin{equation*}
    \left(  \begin{array}{cc} 0& 0\\
g_{A'}^A&0
   \end{array}\right)
,\qquad
    \left(  \begin{array}{cc} g_{B'}^{A'}& 0\\
0&g_B^A
   \end{array}\right)
,\qquad
    \left(  \begin{array}{cc} 0 & g_B^{A'}\\
0&0
   \end{array}\right),
\end{equation*}and by $\mathfrak{e}_\beta^\alpha $ the matrix with all entries zero except for the entry in  $\alpha$-th column and $\beta$-th row to be one.
Then
\begin{equation}\label{eq:brackets0}
 \left [\mathfrak{e}_\beta^\alpha ,\mathfrak{e}_\kappa^\gamma\right]= \delta_\kappa^\alpha\mathfrak{e}_\beta^\gamma -\delta_\beta^\gamma\mathfrak{e}
 _\kappa^\alpha,
\end{equation}
in particular,
\begin{equation}\label{eq:brackets}\begin{split}\left[\mathfrak{e}_{A}^{A'},\mathfrak{e}^{B}_{B'}\right]&=\delta^{A'} _{B'}\mathfrak{e}_{A}^{B}-  \delta_{A}^{B}
\mathfrak{e}^{A'}_{B'} ,\\
\left[\mathfrak{e}_{A}^{A'},\mathfrak{e}^{B}_{C}\right]&=-\delta_{A}^{B} \mathfrak{e}^{A'}_{C} ,\qquad  \left[\mathfrak{e}^{B}_{C}, \mathfrak{e}_{A'}^{A} \right]
=-\delta^{A}_{C} \mathfrak{e}^{ B}_{A'} ,\\
\left[\mathfrak{e}_{A}^{A'},\mathfrak{e}^{B'}_{C'}\right]&= \delta_{C'}^{A'} \mathfrak{e}^{B'}_{A},\qquad   \left[\mathfrak{e}^{B'}_{C'}, \mathfrak{e}_{A'}^{A}
\right] = \delta_{A'}^{B'}\mathfrak{ e}^{A}_{C'}.
 \end{split} \end{equation}
 \begin{rem} The matrix $g$ in \eqref{eq:matix} can be written as $g^\beta_\alpha\mathfrak{e}_\beta^\alpha $.
  The     column  and row   indices of the tuple $(g^\beta_\alpha) $ and that of the basis $ \mathfrak{e}_\beta^\alpha $ are exchanged. We use the  upper indices of the tuple $(g^\beta_\alpha) $ as  row  indices as in differential geometry \cite{GS}.
 \end{rem}

The parabolic subalgebra is
\begin{equation*}
 \mathfrak{p}^{\mathbb{C}}:= \mathfrak{g}_{ 0}^{\mathbb{C}} \oplus \mathfrak{g}_{1}^{\mathbb{C}},
\end{equation*} and let $P^{\mathbb{C}}$ be corresponding subgroup. Then
\begin{equation*}
   {\rm G}_0^{\mathbb{C}}=S\left({\rm GL}( 2,
\mathbb{C})\times {\rm GL}(2n ,
\mathbb{C})\right)= \left({\rm GL}( 2,
\mathbb{C})\times {\rm GL}(2n ,
\mathbb{C})\right)\cap {\rm SL}(2n+2,
\mathbb{C}).
\end{equation*}$\mathbb{C}^{ 2n+2}=\mathbb{C}^{ 2}\oplus \mathbb{C}^{ 2n }$ as the defining representation of $ {\rm SL}(2n+2,
\mathbb{C})$ is a $P^{\mathbb{C}}$-module.
It is obvious that $\mathbb{C}^{ 2}$ in this decomposition is a $P^{\mathbb{C}}$-module, and so is $\mathbb{C}^{ 2n* }$ in the decomposition $\mathbb{C}^{ 2(n+1)*}=\mathbb{C}^{ 2* }\oplus \mathbb{C}^{  2n * }$. They are also $G_0^{\mathbb{C}}$-modules, and $\mathbb{C}^{ 2*}$ and  $\mathbb{C}^{ 2n * }$ are   $G_0^{\mathbb{C}}$-modules dual to  $\mathbb{C}^{ 2}$ and  $\mathbb{C}^{ 2n}$, respectively.
 Then as $G_0^{\mathbb{C}} $-modules,
\begin{equation*}
   \mathfrak{g}_{-1}^{\mathbb{C}}\cong \mathbb{C}^{ 2n  } \otimes \mathbb{C}^{ 2 *},\qquad    \mathfrak{g}_{ 1}^{\mathbb{C}}\cong \mathbb{C}^{ 2  } \otimes \mathbb{C}^{ 2n* }.
\end{equation*}
\subsection{  ${\rm G}/{\rm P}$ }

Let ${\rm G}$ be a real or complex semisimple Lie group and ${\rm P}$ is a parabolic subgroup. A  point of the homogeneous space ${\rm G}/{\rm P}$ is a coset
$h{\rm P}$ for some
$h\in {\rm G}$.
$g\in {\rm G}$ acts on ${\rm G}/{\rm P}$ as
\begin{equation}\label{eq:group-action}
   g.(h{\rm P})= gh{\rm P}.
\end{equation}
  Since for a function    on ${\rm G}/{\rm P}$, the action defined by
\begin{equation*}
   g.f (h{\rm P} )= f \left(g^{-1}h{\rm P} \right)
\end{equation*}
is a group action, i.e.
$
   g_2.( g_1.f)=  ( g_2g_1).f
$, we have to know the action of $g^{-1}$ on the homogeneous space.

In our case, ${\rm G}$ is $ {\rm SL}( n+1,
\mathbb{H})$ or its complexification ${\rm SL}(2n+2, \mathbb{C})$. For an element $g  \in {\rm SL}(2n+2, \mathbb{C})$,   write
\begin{equation}\label{eq:g-1}
   g^{-1}=\left(  \begin{array}{cc}\mathbf{a}_{2\times2}&\mathbf{b}_{2\times2n}\\
\mathbf{c}_{2n\times2}&\mathbf{d}_{2n\times2n}
   \end{array}\right)
\end{equation}
where $\mathbf{a},\mathbf{b},
\mathbf{c}$ and $\mathbf{d}$ are complex matrices. The parabolic subgroup ${\rm P}$ consisting matrices of the form
\begin{equation*}
   \left(  \begin{array}{cc}\mathbf{a}&\mathbf{b}\\
0&\mathbf{d}
   \end{array}\right),
\end{equation*}and
\begin{equation*}
   \left(  \begin{array}{cc} \mathbf{1}_{2  }&0\\
\mathbf{z}& \mathbf{1}_{2 n}
   \end{array}\right){\rm P},\qquad \mathbf{z}\in \mathbb{C}^{2n\times 2},
\end{equation*}
  constitute an open subset of ${\rm G}/{\rm P}$, which is holomorphically diffeomorphic to  $\mathbb{C}^{2n\times 2}$, where $\mathbf{1}_{l  }$ is the $l\times l$
  identity
  matrix.

\begin{prop} The action \eqref{eq:group-action} for $ {\rm SL}(2n+2, \mathbb{C})$ on $\mathbb{C}^{2n\times 2}$ is given by
   \begin{equation} \label{eq:g.z}
\underline{T}_{g^{-1}}:\mathbf{z}\rightarrow  g^{-1}.\mathbf{z} =(\mathbf{c}+\mathbf{d}\mathbf{z} ) (\mathbf{a}+\mathbf{b}\mathbf{z})^{-1},
\end{equation} for $ g^{-1}  $ in \eqref{eq:g-1}.
\end{prop}
\begin{proof}
This is because
\begin{equation}\label{eq:triangular-decomposition}\begin{split}
 &\left(  \begin{array}{cc}\mathbf{a}&\mathbf{b}\\
\mathbf{c}&\mathbf{d}
   \end{array}\right)\left(  \begin{array}{cc}\mathbf{1}_2 &\mathbf{0}\\
 \mathbf{z} &\mathbf{1}_{2 n}
   \end{array}\right)\\=&\left(  \begin{array}{cc}\mathbf{a}+\mathbf{b}\mathbf{z} &\mathbf{b}\\
 \mathbf{c}+\mathbf{d}\mathbf{z}  &\mathbf{d}
   \end{array}\right)\left(  \begin{array}{cc}\mathbf{1}_{2  } &-(\mathbf{a}+\mathbf{b}\mathbf{z})^{-1} \mathbf{b}\\
0&\mathbf{1}_{2 n}
   \end{array}\right) \left(  \begin{array}{cc}\mathbf{1}_{2  } & (\mathbf{a}+\mathbf{b}\mathbf{z})^{-1}\mathbf{b} \\
0&\mathbf{1}_{2 n}
   \end{array}\right) \\
=  &\left(  \begin{array}{cc}\mathbf{a}+\mathbf{b}\mathbf{z}&\mathbf{0}\\
  \mathbf{c}+\mathbf{d}\mathbf{z}    &  \mathbf{d}-(\mathbf{c}+\mathbf{d}\mathbf{z} ) (\mathbf{a}+\mathbf{b}\mathbf{z})^{-1}\mathbf{b}
   \end{array}\right)\left(  \begin{array}{cc}\mathbf{1}_2   & (\mathbf{a}+\mathbf{b}\mathbf{z})^{-1}\mathbf{b} \\
0&\mathbf{1}_{2 n}
   \end{array}\right) \\
=  &\left(  \begin{array}{cc}\mathbf{1}_{2  }&\mathbf{0}\\
 (\mathbf{c}+\mathbf{d}\mathbf{z} ) (\mathbf{a}+\mathbf{b}\mathbf{z})^{-1}& \mathbf{1}_{2 n}
   \end{array}\right)\left(  \begin{array}{cc} J_1(g ^{-1} ,\mathbf{z} ) &\mathbf{0}\\
0  &  {J}_2(g ^{-1} ,\mathbf{z} )
   \end{array}\right)\left(  \begin{array}{cc}\mathbf{1}_{2  } & J_1(g ^{-1} ,\mathbf{z} )^{-1}\mathbf{b} \\
0&\mathbf{1}_{2 n}
   \end{array}\right),
\end{split} \end{equation}
 if we denote
\begin{equation*}\begin{split}
  J_1(g^{-1} ,\mathbf{z} )&:=  \mathbf{a}+\mathbf{b}\mathbf{z}  ,\\
  {J}_2(g^{-1} ,\mathbf{z} )&:= \mathbf{d}-(\mathbf{c}+\mathbf{d}\mathbf{z} ) (\mathbf{a}+\mathbf{b}\mathbf{z})^{-1} \mathbf{b}.
 \end{split} \end{equation*}Then \eqref{eq:triangular-decomposition} mod P gives us the result.
\end{proof}
Since the action \eqref{eq:g.z} is induced from \eqref{eq:group-action}, it is  a group action, i.e.
\begin{equation}\label{eq:g.z2}
 g_2^{-1}.(  g_1^{-1}.\mathbf{z} )=(g_2^{-1}   g_1^{-1}).\mathbf{z}.
\end{equation}

\begin{prop}\label{prop:cocycle}
   $J_\mu(g,\mathbf{z} ) $ is a cocycle,   i.e.
   \begin{equation}\label{eq:cocycle}
    J_\mu(g_2^{-1} g_1^{-1}  ,\mathbf{z} ) = J_\mu\left(g_2^{-1} ,g_1^{-1} .\mathbf{z}\right )J_\mu(g_1^{-1} ,\mathbf{z}),\qquad \mu=1,2.
\end{equation}
\end{prop}
\begin{proof} Denote $g_\alpha^{-1}=\left(  \begin{array}{cc}\mathbf{a}_\alpha &\mathbf{b}_\alpha\\
 \mathbf{c}_\alpha&\mathbf{d}_\alpha
   \end{array}\right)\in {\rm SL}(2n+2, \mathbb{C})$, $\alpha=1,2$. Then, by using \eqref{eq:triangular-decomposition} twice, we get
\begin{equation*}\begin{split}&g_2^{-1} g_1  ^{-1} \left(  \begin{array}{cc}\mathbf{1}_{2 } &0\\
\mathbf{z} &\mathbf{1}_{2 n}
   \end{array}\right)\\
   = & \left(  \begin{array}{cc}\mathbf{a}_2 &\mathbf{b}_2\\
 \mathbf{c}_2&\mathbf{d}_2
   \end{array}\right)\left(  \begin{array}{cc} \mathbf{1}_{2 } &0 \\
 g_1^{-1}.\mathbf{z} & \mathbf{1}_{2 n}
   \end{array}\right)  \left(  \begin{array}{cc}  J_1(g_1^{-1} ,\mathbf{z}) &0\\
 0& J_2(g_1^{-1} ,\mathbf{z})
   \end{array}\right)\left(  \begin{array}{cc}\mathbf{1}_{2 } &  J(g_1^{-1} ,\mathbf{z})^{-1} \mathbf{b}\\
0&\mathbf{1}_{2 n}
   \end{array}\right)\\
=& \left(  \begin{array}{cc}\mathbf{1}_{2 } &0 \\
 g_2^{-1}.(g_1^{-1}.\mathbf{z}) & \mathbf{1}_{2 n}
   \end{array}\right) \left(  \begin{array}{cc}  J_1(g_2^{-1} ,g_1^{-1}.\mathbf{z} ) &0\\
 0&{J}_2(g_2^{-1} ,g_1^{-1}.\mathbf{z} )
   \end{array}\right)\left(  \begin{array}{cc}\mathbf{1}_{2  } & * \\
0&\mathbf{1}_{2 n}
   \end{array}\right)\\
   &\qquad \cdot \left(  \begin{array}{cc}  J_1(g_1^{-1} ,\mathbf{z}) &0\\
 0&{J}_2(g_1^{-1} ,\mathbf{z})
   \end{array}\right)\left(  \begin{array}{cc}\mathbf{1}_{2 } & * \\
0&\mathbf{1}_{2 n}
   \end{array}\right)\\
= &\left(  \begin{array}{cc} \mathbf{1}_{2  } &0 \\
 g_2^{-1}.(g_1^{-1}.\mathbf{z}) & \mathbf{1}_{2 n}
   \end{array}\right) \left(  \begin{array}{cc}  J_1(g_2^{-1} ,g_1^{-1}.\mathbf{z} )   J_1(g_1^{-1} ,\mathbf{z})&0\\
 0&{J}_2(g_2^{-1} ,g_1^{-1}.\mathbf{z} ){J}_2(g_1^{-1} ,\mathbf{z})
   \end{array}\right) \left(  \begin{array}{cc}\mathbf{1}_{2  } & * \\
0&\mathbf{1}_{2 n}
   \end{array}\right)
\end{split} \end{equation*}
This together with the decomposition \eqref{eq:triangular-decomposition} for $(g_1  g_2)^{-1}$  and \eqref{eq:g.z2} implies
the  cocycle condition \eqref{eq:cocycle}.
  \end{proof}

  \eqref{eq:cocycle} means that
$J_\mu$ is a {\it factor of automorphy}.
 The defining representation $  \mathbb{ C}^2$ of ${\rm GL}(2,\mathbb{ C})$ is given by\begin{equation*}
   \mathbf{a}. s_{A'}= \mathbf{a}_{A'}^{B'}s_{B'}
\end{equation*}
for $\mathbf{a}=(\mathbf{a}_{B'}^{A'})\in  \mathfrak {\rm GL}(2,\mathbb{ C})$,  since
\begin{equation*}
   \mathbf{a}. (  \tilde{ \mathbf{a}}  . s_{A'})=\mathbf{a}. (   \tilde{ \mathbf{a}}_{A'}^{B'}s_{B'})=\mathbf{a}_{B'}^{C'} \tilde{
   \mathbf{a}}_{A'}^{B'}s_{C'}=(\mathbf{a}   \tilde{ \mathbf{a}}). s_{A'}
\end{equation*}
 for the other $  \tilde{ \mathbf{a}}  \in   {\rm GL}(2,\mathbb{ C})$. Let $\{\omega^A; {A} =0,\ldots, 2n-1\}$ be a basis of $  \mathbb{ C}^{2n*}$. They are Grassmannian variables, i.e.
$\omega^A\omega^B=-\omega^B\omega^A$.     $\mathbf{d} \in{\rm GL}(2n,\mathbb{ C})$ acts on $  \mathbb{ C}^{2n*}$ as
 \begin{equation*}
  \mathbf{d} . \omega^{A }=  \mathbf{d} _{B}^{ A } \omega^{B }.
\end{equation*}
 This action is not a representation, but $\varrho( \mathbf{d} ) \omega^{A } = \mathbf{d}^{-1} . \omega^{A }$  defines
  the dual representation   of defining representation $  \mathbb{ C}^{2n}$  of ${\rm GL}(2n,\mathbb{ C})$, since
\begin{equation*}
 \varrho(  \mathbf{d} )  \varrho( \tilde{\mathbf{d} }) \omega^{A } =\mathbf{d}^{-1}. (  ( \tilde{\mathbf{d} }^{-1})_{B}^{ A } \omega^{B })= (
 \tilde{\mathbf{d}}^{-1})_{B}^{ A } ( {\mathbf{d} }^{-1})_{C}^{ B } \omega^{C } =(  \mathbf{d}  \tilde{\mathbf{d} })^{-1} . \omega^{A }= \varrho(  \mathbf{d}
 \tilde{\mathbf{d} }) \omega^{A }
\end{equation*}
for $\mathbf{d},\tilde{\mathbf{d} }\in {\rm GL}(2n,\mathbb{ C})$.
 For $j\geq k $, we will denote by  $s^{A'}$   coordinate functions of $  \mathbb{ C}^{2*}$ with the action
\begin{equation*}
   \mathbf{a}. s^{A'}=  \mathbf{a} ^{A'}_{B'}s^{B'}.
\end{equation*}Similarly,  $\varrho(  \mathbf{a} ) \omega^{A } =  \mathbf{a}^{-1} .  s^{A'}$  defines  the dual representation of $  \mathbb{ C}^2$.

For a vector space $V$, let $\Gamma(\mathbb{C}^{2n\times 2}, V)$ be the space of $V$-valued holomorphic functions. Let
\begin{equation}\label{eq:partial-A-A'}
   \partial_{A}^{A' }:=\frac {\partial}{\partial \mathbf{z}_{A' }^{A}}.
\end{equation}Denote
$\omega^{\mathbf{A}}=\omega^{ {A}_1 } \cdots \omega^{ {A}_\tau }$ for a $\tau$-tuple
     $ {\mathbf{A}}= {A}_1 \cdots  {A}_\tau $  for some $\tau$. An element of $\Gamma( \mathbb{C}^{2n\times 2}, \wedge^{\tau}\mathbb{C}^{2n})$ can be written as $f= f_{\mathbf{A} } \omega^{\mathbf{A}}$ with $ f_{\mathbf{A} }$ antisymmetric under permutation of indices.
  Define $\underline d^{A'}  :\Gamma(\mathbb{C}^{2n\times 2}, \wedge^{\tau}\mathbb{C}^{2n})\rightarrow
\Gamma(\mathbb{C}^{2n\times 2},
\wedge^{\tau+1}\mathbb{C}^{2n})$   as
\begin{equation}\label{eq:d}\begin{aligned}&
\underline d^{A'}f:=
\partial_{A}^{A' }f_{\mathbf{A} }~  \omega^A \omega^{\mathbf{A}},
\end{aligned}\end{equation}

  Corresponding  to a notation   on $\mathbb{H}^{ n }$, its counterpart on $\mathbb{C}^{2n\times 2}$  is usually denoted by the same symbol with  underline.

\begin{prop}\label{prop:dd} {\rm  \cite[Proposition 2.2]{wan-wang}} $($1$)$ $ \underline{ d}^{0'}   d^{1'}=-\underline d^{1'}\underline d^{0'}$.\\
$($2$)$ $(\underline d^{0'})^2= (\underline  d^{1'}) ^2=0$.\\
$($3$)$ For $F\in \Gamma(\mathbb{C}^{2n\times 2}, \wedge^{\tau}\mathbb{C}^{2n})$, $G\in \Gamma(\mathbb{C}^{2n\times 2}, \wedge^{\chi}\mathbb{C}^{2n})$, we
have\begin{equation*}\underline d^{A'}(F\cdot G)=\underline d^{A'}
F\cdot
G+(-1)^{\tau}F\cdot \underline d^{A'} G,\qquad  {A'}=0' ,1' .\end{equation*}
\end{prop}
\begin{proof} We give its simple proof here for convenience of readers.  For $F=F_{\mathbf{A} }~ \omega^{\mathbf{A}}$ with $|\mathbf{A}|=\tau$,
\begin{equation*}
  \underline d^{A'}\underline d^{B'}F= \partial_{A}^{A' } \partial_{B}^{B' }F_{\mathbf{A} }~ \omega^A\omega^B\omega^{\mathbf{A}}=-\partial_{B}^{B' }
  \partial_{A}^{A'
  }F_{\mathbf{A} }~
  \omega^B\omega^A\omega^{\mathbf{A}}=-\underline d^{B'}\underline d^{A'}F,
\end{equation*}
and for $G=G_{\mathbf{A} }~ \omega^{\mathbf{A}}$ with $|\mathbf{G}|=\chi$, we have
\begin{equation*}
  \underline d^{A'}(F \wedge  G)= \partial_{A}^{A' }  (F_{\mathbf{A} }G_{\mathbf{B}})~ \omega^A\omega^{\mathbf{A}}\omega^{\mathbf{B}}=\partial_{A}^{A' }
  F_{\mathbf{A}
  }~
  \omega^A\omega^{\mathbf{A}}(G_{\mathbf{B}} \omega^{\mathbf{B}})+  (-1)^{\tau} F_{\mathbf{A} }~\omega^{\mathbf{A}} (\partial_{A}^{A' }G_{\mathbf{B}} \omega^A
  \omega^{\mathbf{B}}).
\end{equation*}The proposition is proved.
\end{proof}

 We will also use the following notations:
  \begin{equation}\label{eq:a-d}
     \mathbf{a}.\underline  d^{A'}=\mathbf{a}_{ B'}^{A'} \underline d^{ B'} ,\qquad \mathbf{d} . \underline  d^{A'}=  \mathbf{d} .\omega^A \partial_{A}^{A' }
  \end{equation}
  for $\mathbf{a}\in \mathfrak{ gl}( 2,\mathbb{C})$, $\mathbf{d}\in \mathfrak{ gl}( 2n,\mathbb{C})$.

The $\sigma$-th symmetric power $\odot^\sigma  \mathbb{C}^{2 }$ is an irreducible ${\rm GL}(2,\mathbb{ C})$-module.
It is convenient to   realize $\odot^{\sigma} \mathbb{C}^{2 }  $ as the space $  \mathcal{P}_\sigma(\mathbb{C}^2)$ of homogeneous polynomials of degree $\sigma$ on
$\mathbb{C}^2$
\cite{LSW}. Denote
       $s_{\mathbf{A}'}:=s_{A_1'}\cdot\ldots\cdot s_{A_\sigma'}$ for $\mathbf{A}' = A_1' \ldots A_\sigma'$. Set $|A_1' \ldots A_\sigma'|=\sigma$. $ \mathbf{a}  .
       s_{\mathbf{A}'}=   \mathbf{a}. s_{A_1'}\cdots   \mathbf{a}. s_{A_\sigma'}  $ for $\mathbf{a} \in  {\rm GL}(2,\mathbb{C})$.
The action of the Lie algebra $\mathfrak {gl}(2,\mathbb{ C})$ is given by
\begin{equation}\label{eq:act-a-s}
  \mathbf{a}  . s_{\mathbf{A}'}=\sum_{j=1}^\sigma  s_{A_1'}\cdots(\mathbf{a}. s_{A_j'})\cdots s_{A_\sigma'}   =\mathbf{a} . s_{ {A}'}\cdot \partial^{
  {A}'}s_{\mathbf{A}'}, \qquad \qquad \partial^{ {A}'}=\frac {\partial}{\partial s_{ {A}'}},
\end{equation}
for $\mathbf{a} \in  \mathfrak {gl}(2,\mathbb{ C})$.
  The  $\tau$-th exterior power $\wedge^{\tau} \mathbb{C}^{2n *} $  is a representation of $\mathfrak {gl}(2n,\mathbb{ C})$ with induced action
  \begin{equation}\label{eq:act-d-omega}
      \mathbf{d}. \omega^{ \mathbf{A}  } =  \sum_{j=1}^\tau \cdots \omega^{A_{j-1} } (\mathbf{d}. \omega^{A_j'}) \omega^{A_{j+1} }\cdots=   \mathbf{d}. \omega^{
      {A}
      }\cdot \partial_A \omega^{ \mathbf{A}  },\qquad \qquad \partial_A=\frac {\partial}{\partial  \omega^{ {A} }}.
  \end{equation}
\begin{rem}
   It is convenient to use derivatives and multiplications with respect to variables $s_{A'}\in\mathbb{C}^{2}$ or Grassmannian  variables
$\omega^{  A}$ to represent linear transformations on the space $\odot^{\sigma}\mathbb{C}^{2}
$ or $  \wedge^\tau\mathbb{C}^{2n*}  $.
\end{rem}

    The   complexified version of   the $k$-Cauchy-Fueter complex is
\begin{equation*}
0\rightarrow \Gamma( \mathbb{C}^{2n\times 2}, \mathcal  {V}_0 )
 \xrightarrow{\underline { \mathcal  {D}}_{0} } \Gamma( \mathbb{C}^{2n\times 2}, \mathcal  {V}_1   )\xrightarrow{\underline {\mathcal  {D}}_{1} }  \cdots
 \xrightarrow{\underline {\mathcal
 {D}}_{2n-2}} \Gamma(
 \mathbb{C}^{2n\times 2},\mathcal  {V}_{2n-1} )\rightarrow0.
 \end{equation*}
A section of $f\in
\Gamma \left( \mathbb{C}^{2n\times 2}, \odot^{\sigma}\mathbb{C}^{2}
\otimes  \wedge^\tau\mathbb{C}^{2n*} \right)$ is a function in complex variables $\mathbf{z}^{A}_{A'}$, $s_{A'}\in\mathbb{C}^{2}$ and Grassmannian  variables
$\omega^{  A}$, which is homogeneous of degree $\sigma$ in $s_{A'}$ and  homogeneous of degree $\tau$ in $\omega^{  A}$.  It is a function  in supervariables as
\begin{equation*}\label{eq:f-supervariables>}
  f(\mathbf{z})= f_{\mathbf{A}}^{ \mathbf{A}'}(\mathbf{z})s_{\mathbf{A}'} \omega^{\mathbf{A} },
\end{equation*}
where $f_{\mathbf{A}}^{ \mathbf{A}'}$ is invariant under permutations of indices $\mathbf{A}'=A_1'\cdots A'_\sigma$ and is  antisymmetric under permutation of indices $\mathbf{A}=A_1 \cdots A _\tau $.
  Define the action for $g\in {\rm SL}(2n+2, \mathbb{C})$ given by \eqref{eq:g-1} as\begin{equation}\label{eq:action-U-j<}\begin{split}
       [\underline  \pi_j(g)f] (\mathbf{z}):=&\frac {f_{\mathbf{A}}^{\mathbf{A}'}(g^{-1}.\mathbf{z})  }{\det (\mathbf{a}+\mathbf{b}\mathbf{z})^{j+1}}
       (\mathbf{a}+\mathbf{b}\mathbf{z})^{-1}   . s_{ \mathbf{{A}}'}\cdot \left[\mathbf{d}    - (\mathbf{c}+\mathbf{d}\mathbf{z} )
       (\mathbf{a}+\mathbf{b}\mathbf{z})^{-1}
       \mathbf{b}
       \right]  . \omega^{ \mathbf{{A}} }
  ,
   \end{split} \end{equation}

\begin{rem}
  If $g$ given by \eqref{eq:g-1} belongs to $G_0^{\mathbb{C}}$, i.e. $g  =\left(  \begin{array}{cc} \mathbf{a}^{-1} &0\\
0 & \mathbf{d} ^{-1}
   \end{array}\right)$, then \eqref{eq:action-U-j<} becomes
   \begin{equation*}
      [\underline  \pi_j(g)f] (\mathbf{z}):= \frac { f_{\mathbf{A}}^{\mathbf{A}'}(\mathbf{d}  \mathbf{z}\mathbf{a}^{-1})  }{ \det (\mathbf{a} )^{j+1} }
       \mathbf{a} ^{-1}   . s_{ \mathbf{{A}}'}\cdot   \mathbf{d}
       . \omega^{ \mathbf{{A}} } .
   \end{equation*}Namely,  $g$ acts on $s_{ \mathbf{{A}}'}$ by the defining representation of $ {\rm GL}( 2,
\mathbb{C})$ and trivially by $ {\rm GL}( 2n,
\mathbb{C})$, meanwhile, it acts on $\omega^{ \mathbf{{A}} } $ by the representation dual to the defining representation  of $ {\rm GL}( 2n,
\mathbb{C})$ and trivially by $ {\rm GL}( 2 ,
\mathbb{C})$. Thus  $\mathcal{V}_j$  as $G_0^{\mathbb{C}}$-module is $ \odot^{k-j} \mathbb{C}^{ 2} \otimes\wedge^j  \mathbb{C}^{ 2n * }  $. While for the case $j\geq k$ in \eqref{eq:reps-j>}, $g$ acts on $s^{ \mathbf{{A}}'}$ by the representation dual to the defining representation of $ {\rm GL}( 2,
\mathbb{C})$ and trivially by $ {\rm GL}( 2n,
\mathbb{C})$.
\end{rem}

   If $j< k
$,
let $\underline  {\mathcal  {D}}_j:= \underline {\mathcal  {D}} $, where
     \begin{equation*}
         \underline {  {\mathcal  {D}}} : =  \partial^{[ A'} \circ\underline  d^{B']}=\frac 12  \partial^{  A'} \circ\underline  d^{B' } -\frac 12   \partial^{ B'} \circ\underline  d^{ A' },
    \end{equation*}is a derivative of second order on functions in variables $\mathbf{z} _{A}^{A'}$, $s^{A'}$ and
$\omega^{  A}$. The operator is zero if $A'= B'$, and the nontrivial one is unique up to a sign for $[A'B']=[0'1']$ or $ [1'0']$.  Here and in the sequel, the antisymmetrisation is explained by the formula above, while $\circ$ is the composition of operators acting on such functions, and is usually omitted.

   If $j= k
$, let
\begin{equation*}
   \underline  {\mathcal  {D}}_k:=\underline  d^{[ A'} \circ \underline  d^{B']} .
\end{equation*}

    If $j \geq k $,
 $ f\in \Gamma(  \mathbb{C}^{2n\times 2},  \odot^{\sigma}\mathbb{C}^{2*}
\otimes  \wedge^\tau\mathbb{C}^{2n*} )  $ ($\sigma=j-k-1,\tau=j+1$) can be written as
\begin{equation*}
  f= f_{\mathbf{A}'\mathbf{A}} s^{\mathbf{A}'} \omega^{\mathbf{A} }
\end{equation*}
with $|\mathbf{A}'|=\sigma$, $| \mathbf{A} |=\tau$. Define
\begin{equation}\label{eq:reps-j>}\begin{split}
       [\underline  \pi_j(g)f] ({\mathbf{z} }):=&\frac {f_{\mathbf{A}'\mathbf{A}} (g^{-1}.\mathbf{z}) }{\det (\mathbf{a}+\mathbf{b}\mathbf{z})^{j+1}}
       \left(\mathbf{a}+\mathbf{b}\mathbf{z} \right)    . s^{
       \mathbf{{A}}'}\cdot \left[\mathbf{d}
       - (\mathbf{c}+\mathbf{d}\mathbf{z} ) (\mathbf{a}+\mathbf{b}\mathbf{z})^{-1} \mathbf{b} \right]  . \omega^{\mathbf{ {A} }}  .
   \end{split} \end{equation}Let $ \underline  { \mathcal  {D}}_j:= \widehat{ \underline{ {\mathcal  D} }}$ with
\begin{equation}\label{eq:widetilde-D}
 \widehat  { \underline{\mathcal  D} }:=s^{ [A'}\circ \underline d^{B']} .
\end{equation} For   $\mathbf{a}\in {\rm GL}( 2,\mathbb{C}) $,
 denote\begin{equation*}
   \mathbf{a}.    \partial^{ A'}:= \mathbf{a}_{ B'}^{A'} \partial^{ B'} , \qquad    \mathbf{a}.    \partial_{ A'}:= \mathbf{a}^{ B'}_{A'} \partial_{ B'}.
\end{equation*}
By the following lemma, $ \partial^{ [A'} \underline d^{B']} $ and $ s^{ [A'} \underline d^{B']} $  are
both invariant under the action of ${\rm SL}(2,\mathbb{C})$,
but   not invariant under the action of ${\rm GL}(2,\mathbb{C})$.

\begin{lem} \label{lem:M-varepsilon} For   $\mathbf{a}\in \mathfrak{ gl}( 2,\mathbb{C}) $, we have
\begin{equation}\label{eq:M-varepsilon}\begin{split}
 \mathbf{a}.  \partial^{ [A'} \circ  \underline  d^{B']} + \partial^{ [A'} \circ\mathbf{a}.\underline d^{ B']}&=\operatorname{tr}\mathbf{a}\,\partial^{ [A'}
 \underline
 d^{ B']},\\\mathbf{a} . s ^{ [ A'} \circ \underline d^{B']}+
 s^{ [A'} \circ\mathbf{a} .\underline d^{ B']} &= \operatorname{tr}  \mathbf{a}  s^{ [A'}  \underline d^{ B']} .
\end{split} \end{equation}
\end{lem}
\begin{proof} Let $\left (\varepsilon^{A'B'}\right) =\left( \begin{array}{cc} 0&
1\\- 1& 0\end{array}\right)$. Then, we have
\begin{equation*}\begin{split}  \mathbf{a}.  \partial^{ [A'} \circ  \underline  d^{B']} + \partial^{ [A'} \circ\mathbf{a}.\underline d^{ B']}& =\mathbf{a}_{ C'}^{A'}\partial^{ [C'} \circ   \underline{d}^{B'] }+\mathbf{a}_{ C'}^{B'}\partial^{[ A'} \circ \underline
d^{ C']}
   \\&=\left(\mathbf{a}_{ C'}^{A'}\varepsilon^{C' B'}-\mathbf{a}_{ C'}^{B'}\varepsilon^{C'A'}\right)  \partial^{ [0'} \circ \underline d^{1'] }
   \\&  =\operatorname{tr}\mathbf{a}\,\partial^{ [A'}  \underline
 d^{ B']},
 \end{split} \end{equation*}by
$\mathbf{a}  \varepsilon-(\mathbf{a}
\varepsilon)^t=\mathbf{a}  \varepsilon+\varepsilon \mathbf{a}
 ^t=\operatorname{tr}\mathbf{a}\varepsilon.
 $
This is because if we write $\mathbf{a}= \left(  \begin{array}{cc} \alpha &\gamma \\ \beta
& \delta
   \end{array}\right)$, then we have
\begin{equation*}\left(  \begin{array}{cc} \alpha &\gamma\\
\beta & \delta
   \end{array}\right)\left(  \begin{array}{cc} 0 &1\\
 -1& 0
   \end{array}\right)  + \left(  \begin{array}{cc} 0 &1\\
 -1& 0
   \end{array}\right)\left(  \begin{array}{cc} \alpha & \beta\\
 \gamma & \delta
   \end{array}\right)
=\left(  \begin{array}{cc} 0 &\alpha +\delta\\
-\alpha -\delta& 0
   \end{array}\right).
\end{equation*}The lemma is proved.
\end{proof}

\begin{prop} For $ f\in \Gamma(  \mathbb{C}^{2n\times 2},  \mathcal{V}_j )  $,
    $\underline{\pi}_j(g_1)\underline{\pi}_j(g_2)f=\underline{\pi}_j(g_1g_2)f$
outside of singularities.
\end{prop}\begin{proof}
Note that if we identify an element  $ f\in \Gamma(  \mathbb{C}^{2n\times 2},  \mathcal{V}_j )  $ with the tuple $(f_{\mathbf{A}\mathbf{A}'}( \mathbf{z}))$,   the
representation $\underline  \pi_j(g)$ in \eqref{eq:action-U-j<} for $j\leq k$ maps the tuple $(f_{\mathbf{A}}^{\mathbf{A}'}( \mathbf{z}))$ to  the tuple $ \left (
[\underline  \pi_j(g)f]_{\mathbf{B } }^{\mathbf{B }' }  (\mathbf{z})\right)$ with
\begin{equation}\label{eq:reps-tuple}
  \left [\underline  \pi_j(g)f\right]_{\mathbf{A}}^{\mathbf{A}'}  (\mathbf{z}):=\frac 1{\det (J_1(g ^{-1} ,\mathbf{z}))^{j+1}} \left [J_1(g ^{-1}
  ,\mathbf{z})^{-1}\right]_{\mathbf{B}'} ^{\mathbf{A}'}J_2(g ^{-1} ,\mathbf{z})  _{\mathbf{A} } ^{\mathbf{B} } f_{\mathbf{B }}^{ \mathbf{B }' }(g^{-1}.\mathbf{z}).
\end{equation}where $\mathbf{B }=B_1 \ldots B_{ j}$, $ \mathbf{B }'=B_1'\ldots B_{k-j}'$, $  \mathbf{A} =A_1 \ldots A_{ j}$, $\mathbf{A}' =A_1'\ldots A_{k-j}'$,
and
\begin{equation*}\begin{split}
\left [J_1(g ^{-1} ,\mathbf{z})^{-1}\right]_{\mathbf{B}'} ^{\mathbf{A}'} &:=  \prod^{k-j}_{\alpha=1}\left [J_1(g ^{-1} ,\mathbf{z})^{-1}\right]_{B_\alpha '}^{
A_\alpha '}\\
  J_2(g ^{-1} ,\mathbf{z})  _{\mathbf{A} } ^{\mathbf{B} } &:=  \prod^{ j}_{\beta=1}   J_2(g ^{-1} ,\mathbf{z}) _{B_\beta }^{ A_\beta }
\end{split} \end{equation*}
Then
\begin{equation*}\begin{split} \left [{\underline {\pi}}_j(g_1){\underline{\pi}}_j(g_2)  )f\right]_{\mathbf{A}}^{\mathbf{A}'}  (\mathbf{z})&=\frac 1{\det (J_1(g_1
^{-1} ,\mathbf{z}))^{j+1}} \left [J_1(g_1 ^{-1} ,\mathbf{z})^{-1}\right]_{\mathbf{B}'} ^{\mathbf{A}'}J_2(g_1 ^{-1} ,\mathbf{z})  _{\mathbf{A} } ^{\mathbf{B} }
[\underline{\pi}_j(g_2)  )f ]_{\mathbf{B }}^{ \mathbf{B }' }(g_1^{-1}.\mathbf{z})\\&=\frac 1{\det (J_1(g_1 ^{-1} ,\mathbf{z}))^{j+1}} \left [J_1(g_1 ^{-1}
,\mathbf{z})^{-1}\right]_{\mathbf{B}'} ^{\mathbf{A}'}J_2(g_1 ^{-1} ,\mathbf{z})  _{\mathbf{A} } ^{\mathbf{B} } \\
 &\qquad\cdot \frac 1{\det (J_1(g_2 ^{-1} ,g_1^{-1}.\mathbf{z}))^{j+1}} \left [J_1(g_2 ^{-1} ,g_1^{-1}.\mathbf{z})^{-1}\right]_{\mathbf{C}'} ^{\mathbf{B}'}
 J_2(g_2^{-1} ,g_1^{-1}.\mathbf{z})_{\mathbf{B} } ^{\mathbf{C}}f _{\mathbf{C}}^{\mathbf{C}'}  \\
 & = \frac 1{\det (J_1(g_2 ^{-1} g_1^{-1}, \mathbf{z}))^{j+1}} \left [J_1(g_2 ^{-1}g_1^{-1} , \mathbf{z})^{-1}\right]_{\mathbf{C}'} ^{\mathbf{A}'}
 J_2(g_2^{-1}g_1^{-1} , \mathbf{z})_{\mathbf{A} } ^{\mathbf{C}}f _{\mathbf{C}}^{\mathbf{C}'} \\
 &=\left [{\underline {\pi}}_j(g_1 g_2)  )f\right]_{\mathbf{A}}^{\mathbf{A}'}  (\mathbf{z}),
\end{split} \end{equation*}by using the cocycle condition \eqref{eq:cocycle}.
 Namely,  $\underline{\pi}_j(g_1)\underline{\pi}_j(g_2)f=\underline{\pi}_j(g_1g_2)f$
outside of singularities.  It is similar for $j> k$.
\end{proof}

\begin{thm}
   $\underline {\mathcal  {D}}_j:\Gamma(\mathbb{C}^{2n\times 2}, \mathcal{ V}_j)\rightarrow \Gamma(\mathbb{C}^{2n\times 2}, \mathcal{  V}_{j+1}) $ is   ${\rm
   SL}(2n+2,\mathbb{C})$-invariant,
   i.e.
\begin{equation*}\underline {\mathcal  {D}}_j \circ \underline {\pi}_j(g) =
  \underline {  \pi}_{j+1}(g) \circ \underline {\mathcal  {D}}_j .
\end{equation*}
for any $g\in {\rm SL}(2n+2,\mathbb{C})$.
\end{thm}
\section{ The  $ {\rm SL}(2n+2, \mathbb{C})$-invariance of the complexified version}
To prove the
 $ {\rm SL}(2n+2, \mathbb{C})$-invariance, it is sufficient to check the invariance under the action of its Lie  algebra
  $ { \mathfrak  sl}(2n+2, \mathbb{C})$ by the following Proposition \ref{thm:DU-UD-X}, which is more easy. See \cite{Jakobsen} for construction of
  invariant operators for ${\rm Mp}(n, R)$ and ${\rm SU}(n, n)$ by this method. Since we do not find appropriate reference for this proposition, we give the proof
  here for convenience of readers.

For   $X\in  \mathfrak{ sl}(2n+2,\mathbb{C})$, consider the subgroup of one parameter $g_t=e^{tX}=I+tX+O(t^2)$. The action of  Lie  algebra is defined as
\begin{equation*}
   d\underline \pi_j (X)f = \left. \frac d{dt} \underline \pi_j (g_t)f\right |_{t=0}.
\end{equation*}

 \begin{prop} \label{thm:DU-UD-X} If
   \begin{equation}\label{eq:DU-UD-X}
      \underline{\mathcal  {D}}_j \circ d\underline \pi_j (X)f = d\underline \pi_{j+1} (X)\circ\underline{\mathcal  {D}}_j f
   \end{equation}
   for any $X\in  \mathfrak{ sl}(2n+2,\mathbb{C})$ and   $f \in\Gamma(\mathbb{C}^{2n\times 2}, \mathcal{ V}_j)  $, then $\underline{\mathcal  {D}}_j$ is  ${\rm
   SL}(2n+2,\mathbb{C})$-invariant.
\end{prop}
\begin{proof}  It is sufficient to prove the
result for $f$ to be a holomorphic polynomial. Let $g_t$ be a subgroup of one parameter generated by $X\in  \mathfrak{
sl}(2n+2,\mathbb{C})$, i.e. $g_t=\exp(tX)$. Differentiate
$
   \underline \pi_j(g_{t +s})f= \underline \pi_j(g_{t }) \underline \pi_j(g_s)f
$
with respect to $s$ at $s=0$ to get
\begin{equation*}
  \frac d{dt} \underline \pi_j(g_{t  })f=\left. \underline \pi_j(g_{t })\frac d{ds} \underline \pi_j(g_s)f\right|_{s=0}=\underline \pi_j(g_{t }) d \underline
  \pi_j(X)f.
\end{equation*}Differentiating it
with respect to $t$ repeatedly, we get
\begin{equation*}
  \frac {d^m}{dt^m} \underline \pi_j(g_{t  })f= \underline \pi_j(g_{t })  (d\underline \pi_j(X))^mf,
\end{equation*}
in particular,
\begin{equation}\label{eq:g-X}
 \left. \frac {d^m}{dt^m} \underline \pi_j(g_{t  })f\right|_{t=0}=    (d\underline \pi_j(X))^mf.
\end{equation}

For fixed $\mathbf{z}\in
\mathbb{C}^{2n\times 2}$, let
\begin{equation*}
   F_1(t)=\underline{\mathcal  {D}}_{j} \underline \pi_j(g_{t  })f(\mathbf{z}), \qquad F_2(t)= \underline \pi_{j+1}(g_{t  })\underline{\mathcal  {D}}_{j}
   f(\mathbf{z}).
\end{equation*}
  Without loss of
generality, we only need to check $F_1(t)\equiv F_2(t)$ for $\mathbf{z}$ near the origin, since both sides are rational function of $\mathbf{z}$. 
Moreover, we only need to check it for $g_t=\exp(tX)$ with $\|X\|$ small, because the identity for $ \exp( X)$ can be obtained by using  the identity for $\exp( X/m) $ for $m$ times.
For such $\mathbf{z}$ and $g$,  $f(g_{t  }^{-1}. \mathbf{z})$ is not singular for $t\in [0,1]$, and so it is a real analytic
curve in $t$. Then  $F_1(t) $ and $F_2(t) $ are both real analytic
functions in $t\in [0,1]$, and
\begin{equation*}
   F_1^{(m)}(0)=\underline{\mathcal  {D}}_{j} (d\underline \pi_j(X))^mf(\mathbf{z}),\qquad  F_2^{(m)}(0)= (d\underline \pi_{j+1}(X))^m\underline{\mathcal  {D}}_{j}
   f(\mathbf{z})
\end{equation*}
by  \eqref{eq:g-X}. But
\begin{equation*}
    \underline{\mathcal  {D}}_{j} (d\underline \pi_j(X))^mf (\mathbf{z}) = (d\underline \pi_{j+1}(X))^m\underline{\mathcal  {D}}_{j} f(\mathbf{z})
\end{equation*}by using the assumed invariance \eqref{eq:DU-UD-X} of $\underline{\mathcal  {D}}_j$ under action of Lie algebra repeatedly. Therefore,
$F_1^{(m)}(0)=  F_2^{(m)}(0)$ for
$m=0,1,\ldots$, and
so
$F_1 \equiv F_2 $ as real analytic functions of $t$. At $t=1$, we get $\underline { \pi}_{j+1}(g) \underline {\mathcal  {D}}_j f(\mathbf{z})=\underline {\mathcal
{D}}_j (\underline
\pi_j(g)f)(\mathbf{z}).$
\end{proof}

  \subsection{Proof of the invariance for the case $j<k $  }
We need to check   \eqref{thm:DU-UD-X} for
\begin{equation*}
   X=\left(  \begin{array}{cc}0&  \mathbf{b}\\
 0& 0
   \end{array}\right),\qquad\left(  \begin{array}{cc} \mathbf{a}&  0\\
 0&  \mathbf{d}
   \end{array}\right) ,\qquad \left(  \begin{array}{cc}0&0\\
 \mathbf{c}  & 0
   \end{array}\right).
\end{equation*}

{\it Case i}.
Since
\begin{equation}\label{eq:g-t-1}
   g_t^{-1}=e^{-tX}=\left(  \begin{array}{cc}\mathbf{1}_2& -t\mathbf{b}\\
 0& \mathbf{1}_{2n}
   \end{array}\right)+O(t^2),\qquad {\rm for}\quad X=\left(  \begin{array}{cc}0&  \mathbf{b}\\
 0& 0
   \end{array}\right),
\end{equation}
we get
   \begin{equation}\label{eq:action-B}\begin{split} g^{-1}_t.\mathbf{z}&= \mathbf{z}  (\mathbf{1}_2-t\mathbf{b} \mathbf{z})^{-1}+O(t^2)= \mathbf{z}+
   t\mathbf{z}\mathbf{b}\mathbf{z}+O(t^2),\\
   \det (\mathbf{1}_2-t\mathbf{b}\mathbf{z})&=1-t\operatorname{tr} (\mathbf{b}\mathbf{z}) +O(t^2).
\end{split}\end{equation}
The {\it infinitesimal vector field} $Y$ of   transformations $g^{-1}_t$ on $\mathbb{C}^{2n\times 2}$  is defined by
   \begin{equation*}
      Yf(z)=\left. \frac d{dt} f  (g^{-1}_t.\mathbf{z})\right |_{t=0}
   \end{equation*}for any holomorphic functions $f$. It follows from \eqref{eq:action-B} that
    \begin{equation}\label{eq:vector-Y}
      Y= (\mathbf{z}\mathbf{b}\mathbf{z})_{B'}^{B} \partial_{B}^{B'}.
   \end{equation}

By differentiation the action \eqref{eq:action-U-j<} of $\underline \pi_j(g)$ and using the Leibnitz law, we
  see that
   \begin{equation}\label{eq:reps-j}\begin{split}
    d\underline \pi_j (X)f= &\left. \frac d{dt} \underline \pi_j (g_t)f\right |_{t=0}\\
    =&   \left. \frac d{dt} \frac {f_{\mathbf{A}}^{\mathbf{A}'}(g_t^{-1}.\mathbf{z})}{\det (\mathbf{1}-t\mathbf{b}\mathbf{z})^{j+1}}
       (\mathbf{1}-t\mathbf{b}\mathbf{z})^{-1}   . s_{ \mathbf{{A}}'}\left(\mathbf{1}   +  \mathbf{z}   (\mathbf{1}-t\mathbf{b}\mathbf{z})^{-1}
       t\mathbf{b}
       \right)  . \omega^{ \mathbf{{A}} }   \right |_{t=0}\\=&
         [ Y + (j+1)\operatorname{tr} ( \mathbf{b}\mathbf{z})]f+
        \mathbf{b}\mathbf{z}    . s_{ {A}'}\cdot \partial^{ {A}'} f +  \mathbf{z} \mathbf{b}   . \omega^{ {A} }\cdot \partial_Af
    \end{split} \end{equation}by \eqref{eq:act-a-s}-\eqref{eq:act-d-omega}.
It follows that
    \begin{equation}\label{eq:DU-UD}\begin{split}
 \underline{\mathcal  {D}} \circ d\underline \pi_j (X)  - d\underline \pi_{j+1} (X)\circ\underline{\mathcal  {D}}   = & \left[\partial^{ [A'} \underline d^{B']}, Y
 +
 (j+1)\operatorname{tr} (\mathbf{b}
 \mathbf{z})\right]   - \operatorname{tr}
 (\mathbf{b} \mathbf{z}) \partial^{ [A'} \underline d^{B']} \\&+
      \left[\partial^{ [A'} \underline d^{B']},( \mathbf{b} \mathbf{z})  .s_{ {A}'}\partial^{A'}  \right] + \left[\partial^{ [A'} \underline d^{B']},
      \mathbf{z}
      \mathbf{b}
       . \omega^{ {A} }\partial_A  \right] . \end{split}  \end{equation}
This is an identity of operators acting on functions in variables $\mathbf{z} _{A}^{A'}$, $s_{A'}$ and
$\omega^{  A}$.
To show it  vanishing, we need to calculate   commutators or anticommutators (i.e. $\{S,T\}=ST+TS$).

\begin{lem} \label{lem:d-Omega}   (1) $ \left\{ \underline d^{A'}  , \mathbf{z} \mathbf{b}   . \omega^{ {A}  }  \right\}   =\omega^A  \Omega^{A'}$, where $\Omega^{A'}:  =  \mathbf{b} _{A}^{A'} \omega^{ A}  
  $;

(2) $\left[\underline d^{A'}, \operatorname{tr} (\mathbf{b}\mathbf{z}) \right] = \Omega^{A'}$

(3) $ \left [\underline d^{A'} , Y \right ]   =         \mathbf{b}\mathbf{z} .\underline d^{ A'} + \mathbf{z}\mathbf{b} .\underline d^{ A'}.$
\end{lem}
\begin{proof}
Noting that \begin{equation*}
   \partial_{A}^{ A' }\mathbf{z}_{ B'}^{B}=\delta_{A}^{ B}\delta_{B'}^{A'},
\end{equation*}by definition,
we get
\begin{equation*}\label{eq:d-BZ-ZB}\begin{split} \left\{ \underline d^{A'}  , \mathbf{z} \mathbf{b}   . \omega^{ {A}  }  \right\} & = \underline d^{A'}\left(    \mathbf{z} \mathbf{b}   . \omega^{ {A}  } \right)=
   \omega^B \partial_{ B}^{ A' } \left(\mathbf{z}^{A}_{ C'}\mathbf{b}^{C'}_{ D}\right) \omega^{ D}
                      =    \omega^A  \mathbf{b}_{ D}^ {A'} \omega^{D}=\omega^A  \Omega^{A'}, \\ \left[\underline d^{A'}, \operatorname{tr} (\mathbf{b}\mathbf{z}) \right] &=\underline d^{A'} \operatorname{tr} (\mathbf{b}\mathbf{z}) =
                      \omega^{
                      A}
                      \partial_{A  }^{    A' }\left(\mathbf{b}_{ D}^{B'}\mathbf{z}_{ B'}^{D}\right)
  =    \mathbf{b}^{A'}_{A} \omega^{ A} = \Omega^{A'},
           \end{split} \end{equation*}
and
\begin{equation*}\label{eq:d-Y}\begin{split}\left[\underline d^{A'} , Y \right]     &=
  \underline d^{  A' }\left(\mathbf{z}\mathbf{b}\mathbf{z}\right)_{ B'}^ {B}\cdot\partial_{B}^{ B'}
                       =     \omega^A \partial_{ A}^{  A' }\left (\mathbf{z}_{ C'}^{B}\mathbf{b}_{ D}^{C'}\mathbf{z}_{ B'}^{D}\right) \partial_{B}^{ B'}    \\
                        &=     \omega^A   \left( \mathbf{b}\mathbf{z}\right)_{  B'}^{A'} \partial_{A}^{ B'}   + \omega^A \left (\mathbf{z}\mathbf{b}\right)_{ A}^
                        {B } \partial_{B}^{ A'} .
           \end{split} \end{equation*}
         The  result follows by using notations \eqref{eq:a-d}. \end{proof}

  For the third   commutator in \eqref{eq:DU-UD},
note that
  \begin{equation}\label{eq:d-BZ-s0}\begin{split}\left[ \underline d^{B'},  \mathbf{b} \mathbf{z}  .s_{ C'}\partial^{C'}  \right]    &= \underline d^{B'}\left(
  \mathbf{b}
  \mathbf{z}
  .s_{ C'}   \right)  \partial^{C'}   =
     \omega^B \partial_{ B}^{ B'} (\mathbf{b}_{ D}^ { D'}\mathbf{z}_{C'}^ {D})  s_ { D'}\partial^{ C'}
                    =  s_ {  D'} \Omega^{ D' }   \partial^{B'},
                      \end{split} \end{equation}by
 $\underline d^{A'}$ commuting $\partial^{ B'}$, and the operator identity for  functions in supervariables:
\begin{equation}\label{eq:operator-identity}
   [UV,W]=UV W- U WV+U W V -W UV=U[ V,W]+[U,W]V.
\end{equation} Then we get
 \begin{equation} \label{eq:d-BZ-s}\begin{split}  \left[\partial^{ [A'} \underline d^{B']}, \mathbf{b} \mathbf{z}   .s_{ C'}\partial^{C'}  \right]   = &\frac 12
\partial^{ A'}\left[ \underline d^{B'},  \mathbf{b} \mathbf{z}   .s_{ C'}\partial^{C'} \right] +\frac 12\left[\partial^{ A'},  \mathbf{b} \mathbf{z}   .s_{
C'}\partial^{C'}
\right]\underline d^{B'}-A'\leftrightarrow B'\\=& \frac 12\partial^{ A'}\circ s_{ C'}\Omega^{C'}   \partial^{B'}
 + \frac 12( \mathbf{b} \mathbf{z})_{C'}^{    A'}   \partial^{C'}  \underline d^{ B'}-A'\leftrightarrow B'\\=&       \Omega^{[A'} \partial^{B']}
 +    \mathbf{b} \mathbf{z}  . \partial^{[A'} \circ \underline d^{B']}
	 \end{split} \end{equation}
since
\begin{equation}\label{eq:partial-s}
  \left[\partial^{ A'},s_ { C'}\right]  = \delta^{ A'} _ { C'},\qquad \left [\partial^{ A'}, \partial^{ B'}\right]=0.
\end{equation}

 Similarly, we have
 \begin{equation}\label{eq:d-BZ-xi0}\begin{split}
    \left[ \underline d^{B'},   \mathbf{z}  \mathbf{b}  . \omega^{ {A} }\partial_A  \right]f& =\omega^A \Omega^{B'  }  \cdot \partial_A f -
      \mathbf{z}  \mathbf{b}   .\omega^{ {A} }\cdot   \underline{d}^{B'}\partial_A f -    \mathbf{z}  \mathbf{b}   .\omega^{ {A} }\partial_A
    \underline
    d^{ B'} f \\
    &= -j    \Omega^{B'} f     -   \mathbf{z}  \mathbf{b}   . \underline d^{ B'}f
    \end{split} \end{equation}
by using the Leibnitz law,  Lemma \ref{lem:d-Omega} (1)  and $\omega^A   \partial_Af=jf$ if $f$ is homogeneous of degree $j$ in $\omega $, and
 \begin{equation*}
   \left \{\underline d^{B'}, \partial_A\right
\}=\underline d^{B'}\circ \partial_A+ \partial_A\circ \underline d^{B'} = \partial_{ A}^{B' }
 \end{equation*}
  as an anti-commutator. Consequently,
  \begin{equation}\label{eq:d-BZ-xi0'}\begin{split}
    \left[\partial^{ [A'} \underline d^{B']}, \mathbf{z}  \mathbf{b}   . \omega^{ {A} }\partial_A  \right]
    &= -j  \partial^{[A'}    \Omega^{B']} - \partial^{ [A'}  \mathbf{z}  \mathbf{b}   . \underline d^{B']},
    \end{split} \end{equation}
 since $\partial^{ A'}$
 commute with other operators.

On the other hand, we have
\begin{equation} \label{eq:d-Y0}\begin{split}  \left[\partial^{ [A'} \underline d^{B']},Y  \right]   = &
\partial^{[ A'}\left[ \underline d^{B']},Y  \right]  =   \partial^{ [A'}   \mathbf{b}\mathbf{z} .\underline d ^{ {B}']   }   +
\partial^{[ A'}  \mathbf{z}\mathbf{b} .\underline d^ { {B}']   }  ,
	 \end{split} \end{equation}by Lemma \ref{lem:d-Omega} (3)  and $\partial^{ A'}$
 commuting $\underline d^{A'}$ and $Y$,
and
\begin{equation}\label{eq:d-tr}\left[\partial^{ [A'} \underline d^{B']},\operatorname{tr} (\mathbf{b}\mathbf{z})\right] =\partial^{ [{A}'}\left[\underline d^{
B']},\operatorname{tr}
(\mathbf{bz})\right]
     =\partial^{[ {A}'}  \Omega^{ {B}']   }   ,
   \end{equation}by Lemma \ref{lem:d-Omega} (2).

Now substituting \eqref{eq:d-BZ-s} \eqref{eq:d-BZ-xi0'}-\eqref{eq:d-tr} into  \eqref{eq:DU-UD}, we see that terms $\partial^{[ {A}'}  \Omega^{ {B}']   }$ and
$\partial^{ [A'}  \mathbf{z}\mathbf{b} .\underline d^{ {B}']   } $ are cancelled each other, respectively. So we get
 \begin{equation*}\begin{split}
 \underline{\mathcal  {D}} \circ d\underline \pi_j (X)  - d\underline \pi_{j+1} (X)\circ\underline{\mathcal  {D}} &  = \partial^{[A'}
 \mathbf{b}\mathbf{z} .\underline d ^{ {B}']   } -
 \operatorname{tr} (\mathbf{b}
 \mathbf{z})\underline{\mathcal  {D}}
 +  \mathbf{b} \mathbf{z}   . \partial^{[A'}   \underline d ^{ {B}']   }   =0\end{split}  \end{equation*}
by using Lemma \ref{lem:M-varepsilon}.

{\it Case ii}.
 Since
   \begin{equation}\label{eq:g-t-2}
      g_t^{-1}=e^{-tX}= \left(  \begin{array}{cc} \mathbf{1}_{ 2} -t\mathbf{a} & 0\\
 0& \mathbf{1}_{2n } -t\mathbf{d}
   \end{array}\right)+O(t^2),  \qquad {\rm for}\quad  X= \left(  \begin{array}{cc}\mathbf{a} & 0\\
 0&\mathbf{d}
   \end{array}\right),
   \end{equation}
      we
have
   \begin{equation*}\begin{split}
    d\underline \pi_j (X)f& = \left.   \frac d{dt}\underline \pi_j\left( g_t\right)\right|_{t=0}f  =    \widetilde{ {Y}} f+(j+1)\operatorname{tr} \mathbf{a} f
      +  \mathbf{a}. s_{ B'}  \cdot \partial^{ B'} f   - \mathbf{d}   . \omega^{ {A} } \cdot \partial_{ {A} } f,
    \end{split} \end{equation*}by differentiate the action \eqref{eq:action-U-j<} of $\underline \pi_j(g_t)$, where
   \begin{equation}\label{eq:widetilde-Y}
     \widetilde{ { Y}}:=\left[(\mathbf{za})^{C}_{ C'}-(\mathbf{dz})^{C}_{ C'}\right] \partial_{C}^{ C'}
   \end{equation}is the infinitesimal vector field of   transformations
   \begin{equation*}
      g_t^{-1}.\mathbf{z}=
 (\mathbf{1}_{2n } -t\mathbf{d})\mathbf{z}(\mathbf{1}_{ 2} -t\mathbf{a} )^{-1}=\mathbf{z}+   t (\mathbf{za} - \mathbf{dz})+O(t^2) .
   \end{equation*}
         Thus
    \begin{equation*}\label{eq:DU-UD-2-j<}\begin{split}
  {\underline{\mathcal  {D}}}\circ d\underline \pi_j (X)   -  d\underline \pi_{j+1} (X)\circ {\underline{\mathcal  {D}}}   =&\partial^{ [{A}'}\left[\underline d^{
  B']},   \widetilde{{Y}}\right ] -\operatorname{tr}
  \mathbf{a}\, \underline{\mathcal  {D}}
        +\left[\partial^{ [{A}'}\underline d^{   B']}, \mathbf{a}. s_{ C'}  \right] \partial^{  C' }
            -\partial^{ [{A}'}\left[\underline d^{  B']}, \mathbf{d}   . \omega^{ {A} } \partial_{ {A} }\right] , \end{split}  \end{equation*}
by $\partial^{ A'}$
 commuting $\partial^{ B'}$, $\underline d^{A'}$ and $\mathbf{d}^t  . \omega^{ {A} } \cdot \partial_{ {A} }$. Here and in the sequel, we use  antisymmetrisation:
 \begin{equation*}
    \partial^{ [{A}'}\left[\underline d^{
  B']},   \widetilde{{Y}}\right ]=\frac 12\partial^{  {A}'}\left[\underline d^{
  B' },   \widetilde{{Y}}\right ]- \frac 12\partial^{B'}\left[\underline d^{
 {A}'  },   \widetilde{{Y}}\right ].
 \end{equation*}
  Note that
   \begin{equation}\label{eq:d-widetildeY}\begin{split}
\left[\underline d^{B'},   \widetilde{ {Y}}\right] &=\omega^{ A}\partial_{A  }^{B'}\left[ \left( \mathbf{ za}\right)^{C}_{ C'}-\left(\mathbf{dz}\right)^{C}_{
C'}\right] \partial_{C}^{ C'} =\omega^{ A} \left(\mathbf{a}_{ C'}^{B'}\partial_{A}^{ C'}- \mathbf{d}_{ A}^{C}\partial_{C}^{B'} \right)
       = \mathbf{a} .\underline d^{B'}- \mathbf{d} . \underline d^{B'},
                  \end{split} \end{equation}
  and
\begin{equation}\label{eq:d-D-xi}\begin{split} \left[\underline d^{B'} ,  \mathbf{d}  .\omega^{ {A} }\partial_A\right]& = -  \mathbf{d}  .\omega^{ {A} }\omega^{
B
} \circ\partial_A \partial_{B}^{ B'}  -  \mathbf{d}  .\omega^{ {A} }\partial_A \circ\omega^{ B } \partial_{B}^{ B'}
 =-  \mathbf{d}  .\omega^{ {A} } \partial_{A}^{ B'}
       = - \mathbf{d}  .\underline d^{B'},
                  \end{split} \end{equation}by using the Leibnitz law and
                 $
                     \left\{\omega^{ B} , \partial_A \right\} =\delta_{A}^{B} .
                 $
                        Now we see that
                  \begin{equation*}\begin{split}
             {\underline{\mathcal  {D}}}\circ d\underline \pi_j (X)   -  d\underline \pi_{j+1} (X)\circ {\underline{\mathcal  {D}}}=  &  \partial^{[ {A}'}
             \mathbf{a} .\underline d^{ {B}']   }- \operatorname{tr} \mathbf{a}\,  \underline{\mathcal  {D}} +  \mathbf{a}_{C'}^{[ A'}\underline d^{   B']}
             \partial^{ C'   }
        \\
        =  &  \partial^{ [{A}'} \mathbf{a} .\underline d^{ {B}']   }- \operatorname{tr} \mathbf{a}\,  \underline{\mathcal  {D}} +
     \mathbf{a} .
     \partial^{ [ {A}'} \underline d^{   B ']   }  =0
                   \end{split}  \end{equation*}
            by  using  \eqref{eq:partial-s} and Lemma     \ref{lem:M-varepsilon} again.

{\it Case iii}.  The last case is trivial, since it is direct to see that $\underline \pi_j (X)f(\mathbf{z})=f(\mathbf{z}+\mathbf{c})$.
 \subsection{Proof of the invariance for  the case   $j> k $  }

{\it Case  i}.
Differentiate \eqref{eq:reps-j>} for
      $g_t^{-1}$ given by \eqref{eq:g-t-1}
  to get
   \begin{equation*}\begin{split}
    d\underline \pi_j (X)f =&\left.   \frac d{dt}\underline \pi_j\left( e^{-tX}\right)\right|_{t=0} f =
         [ Y+ (j+1)\operatorname{tr} ( \mathbf{b}\mathbf{z})]f
-  \mathbf{b}\mathbf{z}   . s^{ {A}'}  \cdot   \partial _{ {A}'} f +    \mathbf{z} \mathbf{b}    . \omega^{ {A} }\cdot
        \partial_Af,
    \end{split} \end{equation*}by \eqref{eq:action-B},
 where $Y$ is given by \eqref{eq:vector-Y}. Then we have
    \begin{equation}\label{eq:d-DU-UD>}\begin{split}
\underline {\widehat{\mathcal D}} \circ d\underline \pi_j (X)   - d\underline \pi_{j+1} (X)\circ\underline{ \widehat{\mathcal  D}}   = &\left[\underline {\widehat{\mathcal D}}, Y\right]  +
(j+1)\left[\underline {\widehat{\mathcal D}},
\operatorname{tr}
 (\mathbf{b}\mathbf{z})\right]- \operatorname{tr} (\mathbf{b}\mathbf{z})\underline {\widehat{\mathcal D}}   \\&- \left[ \underline{ \widehat{\mathcal
 D}}, \mathbf{b}\mathbf{z}  . s^{ {A}'}  \cdot  { \partial}_{ {A}'}\right]+\left[\underline{\widehat{\mathcal D}},   \mathbf{z} \mathbf{b}
  . \omega^{ {A}  }\cdot \partial_A\right ]  .\end{split}  \end{equation}

Note  that
\begin{equation}\label{eq:d-Y-widetilde} \begin{split} \left[\underline{\widehat{\mathcal D}}, {Y}  \right]   = &
s^{[ A'}\left[ \underline d^{   B ']   }, Y   \right]  =  s^{ [A'}   \mathbf{b}\mathbf{z} .\underline d^{   B ']   } + s^{ [A'}  \mathbf{z}\mathbf{b} .\underline
d^{   B ']   }
	 \end{split} \end{equation}by Lemma \ref{lem:d-Omega} (3), $s^{ A'}$ commuting $Y$ and $ \underline d^{B'}$,   and
\begin{equation}\label{eq:d-tr-widetilde}\left[\underline {\widehat{\mathcal D}},\operatorname{tr} (\mathbf{b}\mathbf{z})\right] =s^{ [{A}'}\left[\underline d^{   B ']
},\operatorname{tr}
(\mathbf{bz})\right]
     =s^{ [{A}'} \Omega^{   B ']   }   .
   \end{equation}
By  \eqref{eq:d-BZ-xi0}, we have
 \begin{equation}\label{eq:d-BZ-xi0-widetilde}\begin{split}\left[\underline{\widehat{\mathcal D}},   \mathbf{z}  \mathbf{b}  . \omega^{ {A} }\partial_A
 \right]&=
    s^{ [A'}\left[ \underline d^{   B ']   },   \mathbf{z}  \mathbf{b}   . \omega^{ {A} }\partial_A  \right]
     = -(j+1)   s^{ [{A}'} \Omega^{   B ']   }   -s^{ [A'}     \mathbf{z}  \mathbf{b}    .\underline d^{   B ']   }.
    \end{split} \end{equation}
  Here $\omega^A   \partial_Af=(j+1)f$ for $ f\in \Gamma( \mathbb{C}^{2n\times 2}, \mathcal V_j )  $   homogeneous of degree $j+1$ in $\omega $.
Since\begin{equation*}\label{eq:d-BZ-S0}\begin{split}\left[ \underline d^{B'}, \mathbf{b}\mathbf{z}  . s^{ C'} {  \partial}_{C'}  \right]&=
\underline d^{ B'}\left( \mathbf{b}\mathbf{z}  .s^{C'}  \right) {  \partial}_{C'}    =
     \omega^B \partial_{ B}^{ B' } \left(\mathbf{b}_{ D}^ {C'}\mathbf{z}_{ D'}^{D}\right)  s^{ D'} {  \partial}_{C'}
                     =    s^{ B' } \Omega^{C' }{  \partial}_{C'} ,
                     \end{split} \end{equation*}
                   we get
\begin{equation} \label{eq:d-BZ-S}\begin{split} \left[ \underline{\widehat{\mathcal D}}, \mathbf{b}\mathbf{z}   .s^{ C'} \partial _{C'}    \right]   =
&
 s^{[A'  }\left[ \underline d^{   B ']   }, \mathbf{b}\mathbf{z}  .s^{ C'} \partial _{C'}      \right] +\left[ s^{[A'
 }, \mathbf{b}\mathbf{z}
 .s^{ |C'|} \partial _{C'}    \right]\underline d^{   B ']   }\\=&s^{[A'  }s^{   B ']   } \Omega^{C'} \partial _{C'}
-  \mathbf{b}\mathbf{z}  .s^{ [A'}   \underline d^{B']}\\= &-   \mathbf{b}\mathbf{z}  .s^{ [A'}   \underline d^{B']} ,
	 \end{split} \end{equation}
by \eqref{eq:operator-identity}.

Now substitute \eqref{eq:d-Y-widetilde}-\eqref{eq:d-BZ-S} into \eqref{eq:d-DU-UD>} to get
  \begin{equation*}\begin{split}
 \underline{\widehat{\mathcal D}}\circ  d\underline \pi_j (X)   - d\underline \pi_{j+1} (X)\circ \underline{\widehat{\mathcal D}}
  =    & s^{ [A'}    \mathbf{b}\mathbf{z} .\underline d^{B']}+   \mathbf{b}\mathbf{z}  .s^{ [A'}   \underline d^{B']}-   \operatorname{tr}
  \left(\mathbf{b}\mathbf{z}\right)\underline {\widehat{\mathcal D}}
 =0  \end{split}  \end{equation*}by using Lemma     \ref{lem:M-varepsilon} again.

 {\it Case  ii}.  Differentiate \eqref{eq:reps-j>} for $g_t^{-1}=e^{-tX}$ with $
      X = \left(  \begin{array}{cc} \mathbf{a} & 0\\
 0& \mathbf{d}
   \end{array}\right)
$
  to get
      \begin{equation*}\begin{split}
   [d\underline \pi_j (X )f](\mathbf{z})=& \left.   \frac d{dt}\underline \pi_j\left( e^{-tX}\right)\right|_{t=0}f(\mathbf{z})
  =     \widetilde{{Y}}f   +(j+1)\operatorname{tr} \mathbf{a}f
      - \mathbf{a} . s^{ {A}'} \cdot     {\partial}_{ {A}'} f-  \mathbf{d} .\omega^{ {A} } \cdot \partial_Af,
    \end{split} \end{equation*}
      where $\widetilde{ {Y}}$ is given by \eqref{eq:widetilde-Y}.
      Thus
    \begin{equation*}\begin{split}
 \underline{\widehat{\mathcal D}}\circ d\underline \pi_j (X )  -  d \underline \pi_{j+1} (X )\circ\underline {\widehat{\mathcal D}}  = &\left[\underline{\widehat{\mathcal D}} ,
 \widetilde{ {Y}}\right] -\operatorname{tr}
 (\mathbf{a})\underline{\widehat{\mathcal D}}
  -\mathbf{a} . s^{ {A}'}\left[\underline{\widehat{\mathcal D}} ,   {\partial}_{ {A}'} \right] +\mathbf{d} .\omega^{ {A} }\cdot \left\{\underline{\widehat{\mathcal D}}  ,
  \partial_A\right\} \\
=  &s^{[ A'} \mathbf{a} .\underline d^{B']}-s^{ [A'} \mathbf{d} .\underline d^{B']}  -\operatorname{tr} (\mathbf{a})\underline{\widehat{\mathcal D}}      +\mathbf{a} .  s^{[
    A'}\partial_{A}^{B']} + s^{[ A'} \mathbf{d} .\underline d^{B']}\\
=  & s^{ [A'} \mathbf{a} .\underline d^{ B']} -\operatorname{tr} (\mathbf{a})\underline{\widehat{\mathcal D}}      +\mathbf{a} . s ^{ [ A'} \underline d^{B']}   =0
\end{split}
\end{equation*}
    by using \eqref{eq:d-widetildeY},
    \begin{equation*}
      \left \{\underline{\widehat{\mathcal D}}  , \partial_A\right\}=\underline{\widehat{\mathcal D}}   \partial_A+
    \partial_A\underline{\widehat{\mathcal D}}=s^{[
    A'}\partial_{A}^{B']},
    \end{equation*}
    and  Lemma     \ref{lem:M-varepsilon} again.

{\it Case iii}.  This is trivial.

   \subsection{Proof of the invariance for the case  $k=j$  }  {\it Case  i}.  For $X=\left(  \begin{array}{cc}0& \mathbf{b}\\
 0& 0
   \end{array}\right)$, differentiate representations \eqref{eq:action-U-j<} for $j=k$ and \eqref{eq:reps-j>} for $j=k+1$
  to get
     \begin{equation*}\begin{split}
   [d\underline \pi_{k} (X)f](\mathbf{z})=& \left.   \frac d{dt}\underline \pi_{k}\left( e^{-tX}\right)\right|_{t=0}f =
         [  {Y} + (k+1)\operatorname{tr} (\mathbf{b}\mathbf{z})]f
         +    \mathbf{z}  \mathbf{b}  . \omega^{ {A} }\partial_Af  ,
   \\
   [d\underline \pi_{k+1} (X)F](\mathbf{z})=& \left.   \frac d{dt}\underline \pi_{k+1}\left(  e^{-tX}\right)\right|_{t=0}F=
         [  {Y} + (k+2)\operatorname{tr} (\mathbf{b}\mathbf{z})]F
        +   \mathbf{z}  \mathbf{b}   . \omega^{ {A} }\partial_A F,
    \end{split} \end{equation*}for $f\in \Gamma( \mathbb{C}^{2n\times 2}, \mathcal  {V}_k)$, $F\in \Gamma( \mathbb{C}^{2n\times 2}, \mathcal  {V}_{k+1})$.
    So we have \begin{equation}\label{eq:triangle-U}\begin{split}
\underline{\mathcal D}_k \circ d\underline \pi_{k} (X) -  d\underline \pi_{k+1} (X)\circ \underline{\mathcal D}_k =&
    \left [ \underline{\mathcal D}_k , Y + (k+1)\operatorname{tr} (\mathbf{b}\mathbf{z})\right]
   -\operatorname{tr} (\mathbf{b}\mathbf{z}) \underline{\mathcal D}_k
 +\left[ \underline{\mathcal D}_k ,   \mathbf{z}  \mathbf{b}   . \omega^{ {A} }\partial_A  \right]  .
    \end{split} \end{equation}
Note that $\underline{\mathcal D}_k=\underline d^{0'}\underline d^{1'}$ and
    \begin{equation*}\begin{split}
    \left[\underline{\mathcal D}_k,    \mathbf{z}  \mathbf{b}   . \omega^{ {A} }\partial_A  \right]&=\underline d^{0'}\left[\underline
    d^{1'} ,   \mathbf{z}
    \mathbf{b}
     . \omega^{ {A} }\partial_A \right]+\left[\underline d^{0'} ,    \mathbf{z}  \mathbf{b}   . \omega^{ {A} }\partial_A \right]\underline
    d^{1'}\\
 &=-k   \underline d^{0'} \circ\Omega^{1'}      - \underline d^{0'}\circ    \mathbf{z}  \mathbf{b}   .  \omega^{ {A} }\partial_A^{1'}-(k +1)
 \Omega^{0'} \underline
 d^{1'}     -
    \mathbf{z}
 \mathbf{b}   .\underline d^{0'}\circ\underline d^{1'}\\
 &= k    \Omega^{1'} \underline d^{0'}     -  \omega^{ {A} } \Omega^{0'}\partial_{ A}^{1' }+  \mathbf{z}  \mathbf{b}   .\underline d^{ 1'}
 \circ\underline d^{0'} -(k +1)   \Omega^{0'} \underline d^{1'}     -
   \mathbf{z}
 \mathbf{b}   .\underline d^{ 0'}\circ\underline d^{1'}\\& = - 2k     \Omega^{
   [ 0' }    d^{1']}    -  2  \mathbf{z}  \mathbf{b}   .\underline d^{[ 0'}\circ  d^{1']}
    \end{split} \end{equation*}
  by using \eqref{eq:d-BZ-xi0} repeatedly, Lemma \ref{lem:d-Omega} (1) and
  $\underline d^{0'} (\Omega^{1'}f)=-\Omega^{1'}\underline d^{0'}f $. We also have
            \begin{equation*}\begin{split}\left[\underline{\mathcal D}_k , Y \right]= &\underline d^{0'}\left[\underline d^{1'} , Y
            \right]+\left[\underline d^{0'} , Y \right]\underline
            d^{1'}
     \\
= & \left [\underline d^{0'},(\mathbf{b}\mathbf{z})^{1'}_{  B'} \right ] \underline d^{ B'} +  (\mathbf{b}\mathbf{z})^{1'}_{  B'} \underline d^{0'} \underline d^{
B'}
+\left\{\underline d^{0'},
 \mathbf{z}\mathbf{b} .\omega^B  \right\}\partial_{B }^{1'}
-  \mathbf{z}\mathbf{b} .\omega^B  \partial_{B }^{1'}\underline d^{0'}\\&+( \mathbf{b}\mathbf{z})^{0'}_{  B'} \underline d^{ B'} \underline d^{1'} +
 \mathbf{z}\mathbf{b} .\underline d^{ 0'}\circ\underline d^{1'}
\\
          = & \Omega^{ 1' }    \underline d^{0'} +\operatorname{tr}  (\mathbf{b}\mathbf{z})  \underline d^{0'} \underline d^{ 1'}-\Omega^{0'} \underline d^{1'}
         -
          \mathbf{z}\mathbf{b} .\underline d^{ 1'} \circ\underline d^{0'} + \mathbf{z}\mathbf{b} .\underline d^{ 0'}\circ\underline d^{1'}\\
          = &  - 2 \Omega^{
   [ 0' }    d^{1']}     + \operatorname{tr}  (\mathbf{b}\mathbf{z}) \underline{\mathcal D}_k
             +2  \mathbf{z}  \mathbf{b}   .\underline d^{[ 0'}\circ  d^{1']}
           \end{split} \end{equation*}
     by using Lemma \ref{lem:d-Omega} (1) (3)  and anti-commutativity of $\underline d^{A'}$'s in   Proposition \ref{prop:dd}, and
 \begin{equation*}\begin{split}\left[\underline{\mathcal D}_k, \operatorname{tr}  (\mathbf{b}\mathbf{z})  \right]= &\underline d^{0'}\left[\underline
 d^{1'} , \operatorname{tr}
 (\mathbf{b}\mathbf{z})  \right]+\left[\underline d^{0'} , \operatorname{tr}  (\mathbf{b}\mathbf{z}) \right]\underline d^{1'}
    \\
= &  \underline d^{0'} \circ \Omega^{ 1' }+\Omega^{ 0' }\circ \underline d^{1'}=2 \Omega^{
   [ 0' }    d^{1']}   ,
           \end{split} \end{equation*}
by Lemma \ref{lem:d-Omega} (2). Substitute the above three identities into \eqref{eq:triangle-U} to see    its  vanishing.

 {\it Case  ii}.
       For $X= \left(  \begin{array}{cc} \mathbf{a} & 0\\
 0& \mathbf{d}
   \end{array}\right)$,
     \begin{equation*}\begin{split}
   [d\underline \pi_{k} (X)f](z)=&
         [ \widetilde{{Y} }+ (k+1) \operatorname{tr} (\mathbf{a})]f- \mathbf{d}  . \omega^{ {A} }\partial_Af
        ,
    \\
   [d\underline \pi_{k+1} (X)F](z)=&
         [ \widetilde{{Y} }+ (k+2)\operatorname{tr} (\mathbf{a})]F
        - \mathbf{d} .\omega^{ {A} }\partial_AF ,
    \end{split} \end{equation*}
    where $ \widetilde{{Y}}$ is given by \eqref{eq:widetilde-Y}.
   Then
   \begin{equation*}
   \underline{\mathcal D}_k \circ d\underline \pi_j(X)   - d\underline \pi_{j+1}(X)\circ \underline{\mathcal D}_k =\left[\underline{\mathcal D}_k   , \widetilde{
   Y}  \right] -
     \operatorname{tr} (\mathbf{a}) \underline{\mathcal D}_k
      -\left[\underline{\mathcal D}_k ,
      \mathbf{d} .\omega^{ {A} }\partial_A\right]   =0,
   \end{equation*}
since
\begin{equation*}\begin{split} \left[ \underline{\mathcal D}_k,  \widetilde{ { Y}}\right] & =  \underline d^{0'}\left[\underline d^{1'} ,  \widetilde{ Y
}\right]+\left[\underline d^{0'} ,
\widehat{Y } \right]\underline d^{1'}
     \\&= \underline d^{0'}\circ \mathbf{a}.\underline d^{ 1'} -  \underline d^{0'}\circ\mathbf{d} .\underline d^{1'} +
     \mathbf{a}.\underline d^{
     0'} \underline d^{1'}- \mathbf{d} .\underline d^{0'} \underline d^{1'}\\&=  \operatorname{tr} \mathbf{a}\,
     \underline d^{ 0'} \underline d^{1'}-\underline d^{0'} \circ  \mathbf{d} .\underline d^{ 1'} -  \mathbf{d} .\underline d^{ 0'}\circ\underline d^{1'}
                  \end{split} \end{equation*}
by \eqref{eq:d-widetildeY}  and anti-commutativity of $\underline d^{A'}$'s, and
\begin{equation*}\begin{split} \left[ \underline{\mathcal D}_k,  \mathbf{d}  .\omega^{ {A} }\partial_A\right]& =\underline d^{0'}\left[\underline d^{1'},
\mathbf{d}  .\omega^{
{A}
}\partial_A\right]+\left[\underline d^{0'},  \mathbf{d}  .\omega^{ {A} }\partial_A\right]\underline d^{1'}
     \\& = -\underline d^{0'}\circ\mathbf{d}  .\underline d^{ 1'} -\mathbf{d}  .\underline d^{ 0'} \underline d^{1'}
                  \end{split} \end{equation*}
by \eqref{eq:d-D-xi}.

{\it Case iii}.  This is trivial.

 \section{The  invariance  on   $\mathbb{H}^{ n }$ and complexes on locally   projective flat manifolds  }
\subsection{ The   invariance  on     $\mathbb{H}^{ n }$}

Let $\mathbf{a}=(\mathbf{a}_{jk})_{p\times m}$ be a quaternionic $(l\times m)$-matrix and write
$
    \mathbf{a}_{jk}= {a}_{jk}^1+\textbf{i} {a}_{jk}^2+\textbf{j} {a}_{jk}^3+\textbf{k} {a}_{jk}^4\in\mathbb{H}.
$
We define $\tau(\mathbf{a})$ to be the complex $(2p\times 2m)$-matrix \begin{equation}\label{tau}\tau(\mathbf{a})=\left(
                                                              \begin{array}{cccc}
                                                                \tau(\mathbf{a}_{00}) & \tau(\mathbf{a}_{01}) & \cdots & \tau(\mathbf{a}_{0(m-1)}) \\
                                                                \tau(\mathbf{a}_{10}) & \tau(\mathbf{a}_{11}) & \cdots & \tau(\mathbf{a}_{1(m-1)}) \\
                                                                \cdots & \cdots & \cdots & \cdots \\
                                                                \tau(\mathbf{a}_{(p-1)0}) & \tau(\mathbf{a}_{(p-1)1}) & \cdots & \tau(\mathbf{a}_{(p-1)(m-1)}) \\
                                                              \end{array}
                                                            \right),
\end{equation}where $\tau(\mathbf{a}_{jk})$ is the complex $(2\times2)$-matrix
\begin{equation}\label{2.301}\left(
                                       \begin{array}{cc}
                                         a_{jk}^1+\textbf{i}a_{jk}^2 & -a_{jk}^3-\textbf{i}a_{jk}^4 \\
                                         a_{jk}^3-\textbf{i}a_{jk}^4 & \quad a_{jk}^1-\textbf{i}a_{jk}^2 \\
                                       \end{array}
                                     \right).\end{equation}
                                      This is motivated by the   embedding of quaternionic  numbers into $2\times 2$-matrices  \cite{wan-wang} \cite{Wa10}.

\begin{prop}\label{prop:embed} {\rm   \cite[Proposition 2.1]{wan-wang}} $(1)$ $\tau(\mathbf{ab})=\tau(\mathbf{a})\tau(\mathbf{b})$ for a quaternionic $(p\times
m)$-matrix $\mathbf{a}$ and a quaternionic $(m\times l)$-matrix $\mathbf{b}$. In particular, for $\mathbf{q}'=\mathbf{a}\mathbf{q} $ with $\mathbf{q},\mathbf{q}'
\in\mathbb{H}^n$  and a quaternionic $(n\times
n)$-matrix  $\mathbf{a} $, we have
\begin{equation}\label{eq:Aq}
 \tau(\mathbf{q}')=\tau(\mathbf{a})\tau(\mathbf{q})
\end{equation}
 as complex $(2n\times2)$-matrices.
\end{prop}

By Proposition \ref{prop:embed},  $\tau$ is  an isomorphism from $ \mathfrak
  {sl}(n+1,\mathbb{H}) $ to a subalgebra of $\mathfrak {sl}(2n+2,\mathbb{C})$, and so
is  an isomorphism from $  {\rm {SL}}(n+1,\mathbb{H}) $ to a subgroup of $ {\rm {SL}}(2n+2,\mathbb{C})$. By Proposition \ref{prop:embed},
 $ \mathbb{ C}^2$ and  $ \mathbb{ C}^{2*}$  have the actions of $ {\rm GL}(1,\mathbb{ H}) $ given by $ \mathbf{ q}. s_{A'}: =  \tau(\mathbf{q}).s_{A'}$
and
$
  \mathbf{ q}. s^{A'}:=  \tau(\mathbf{q}). s^{A'}$, respectively, and     $ \mathbb{ C}^{2n*}$  have the action  of $ {\rm GL}(2n,\mathbb{ H}) $ given by $ \mathbf{ d}. \omega^{A }: =  \tau(\mathbf{d}).   \omega^{A }$.
  By embedding $\tau$, we have
\begin{equation}\label{eq:q-z}\tau(\mathbf{q})=  \left( \mathbf{z}_{A'}^{A}
                           \right):=\left(
                                      \begin{array}{rr}
                                      x_{0}+\textbf{i}x_{1} & -x_{2}-\textbf{i}x_{3} \\x_{2}-\textbf{i}x_{3} & x_{0}-\textbf{i}x_{1} \\
                                                                               \vdots\qquad&\vdots\qquad\\
                                         x_{4l}+\textbf{i}x_{4l+1} & -x_{4l+2}-\textbf{i}x_{4l+3} \\
                                            x_{4l+2}-\textbf{i}x_{4l+3} &  x_{4l }-\textbf{i}x_{4l+1} \\
                                          \vdots\qquad&\vdots\qquad
                                      \end{array}
                                    \right),\qquad  \text{ for }\quad  \mathbf{q}=\left(
                                      \begin{array}{c }
                                      \mathbf{q}_{1}  \\\mathbf{q}_{2}  \\
                                                                               \vdots \\
                                        \mathbf{ q}_{n}
                                      \end{array}
                                    \right),
\end{equation}where  $\textbf q_l=x_{4l}+\textbf{i}x_{4l+1} +\textbf{j} x_{4l+2}+\textbf{k}x_{4l+3} $, $l=0,\ldots,n-1$. $\tau(\mathbb{H}^n)$ is a $4n$-dimensional
totally real subspace of $\mathbb{C}^{2n\times 2}$.
Note that we have the inverse $\tau^{-1}:\tau(\mathbb{H}^n)\rightarrow  \mathbb{H}^n$. By applying $\tau^{-1} $ to \eqref{eq:g.z}, we get
  the  fractional linear    action  \eqref{eq:g.q} on $\mathbb{H}^{ n }$ for $
   g^{-1}\in {\rm
 {SL}}(n+1,\mathbb{H})$ in  \eqref{eq:g-H}.
The action \eqref{eq:g.q} a group action:
\begin{equation}\label{eq:g.q2}
 g_2^{-1}.(  g_1^{-1}.\mathbf  q )=(g_2^{-1}   g_1^{-1}).\mathbf  {q}
\end{equation}for $g_1,g_2  \in{\rm
 {SL}}(n+1,\mathbb{H}) $
 by applying $\tau^{-1}$ to \eqref{eq:g.z2}.

Restricted to $\tau(\mathbb{H}^n)$,
derivatives $\frac {\partial  }{\partial \mathbf{z}_{A}^{ A' }}$ can be realized as
the
first-order differential operators
\begin{equation}\label{eq:nabla-x}\begin{split}\left(
                               \nabla_{A}^{A' }
                           \right):&=\frac 12\overline{\left(
                                      \begin{array}{rr}
                                      \partial_{x_{0}}+\textbf{i}\partial_{x_{1}} & -\partial_{x_{2}}-\textbf{i}\partial_{x_{3}} \\
                                        \partial_{x_{2}}-\textbf{i}\partial_{x_{3}} & \partial_{x_{0}}-\textbf{i}\partial_{x_{1}} \\
                                         \vdots\qquad&\vdots\qquad\\
                                         \partial_{x_{4l}}+\textbf{i}\partial_{x_{4l+1}} & -\partial_{x_{4l+2}}- \textbf{i}\partial_{x_{4l+3}} \\
                                        \partial_{x_{4l+2}}-\textbf{i}\partial_{x_{4l+3}} & \partial_{x_{4l}}-\textbf{i}\partial_{x_{4l+1}} \\
                                        \vdots\qquad&\vdots\qquad
                                      \end{array}
                                    \right)}^t\\
                                    &=\frac 12 \left(
                                      \begin{array}{rrrrrr}
                                      \partial_{x_{0}}-\textbf{i}\partial_{x_{1}} &\partial_{x_{2}}+\textbf{i}\partial_{x_{3}} &\cdots & \partial_{x_{4l}}-\textbf{i}\partial_{x_{4l+1}} & -\partial_{x_{4l+2}}+ \textbf{i}\partial_{x_{4l+3}}  &\cdots\\
                                      -  \partial_{x_{2}}+\textbf{i}\partial_{x_{3}} &- \partial_{x_{0}}+\textbf{i}\partial_{x_{1}} &\cdots   &
                                        \partial_{x_{4l+2}}+\textbf{i}\partial_{x_{4l+3}} & \partial_{x_{4l}}+\textbf{i}\partial_{x_{4l+1}}  &\cdots
                                                                             \end{array}                                    \right) .
 \end{split} \end{equation}

 Now define $d^{A'}  :\Gamma(\mathbb{H}^{ n}, \wedge^{\tau}\mathbb{C}^{2n})\rightarrow \Gamma(\mathbb{H}^{ n},
\wedge^{\tau+1}\mathbb{C}^{2n})$   as
\begin{equation}\label{eq:d-real}\begin{aligned}&
d^{A'}F:=
   \nabla_{A}^{ A' }f_{\mathbf{A} }~\omega^A\omega^{\mathbf{A}},
\end{aligned}\end{equation}
for $F= f_{\mathbf{A} } \omega^{\mathbf{A}}\in \Gamma( \mathbb{H}^{ n}, \wedge^{\tau}\mathbb{C}^{2n})$. Denote
$
  \triangle u=d^{0'} d^{1'}u .
$
Operators in   the  $k$-Cauchy-Fueter complex \eqref{eq:CF-complex} are given by
 \begin{equation*}   \mathcal{ D}_j
 = \left\{
    \begin{array}{ll}     \partial^{[A'}   d^{B']} ,\quad &{\rm if}\quad j=0, \ldots, k-1,\\
      d^{[A'}  d^{B']} ,\quad &{\rm if}\quad j=  k ,\\
    s^{[A'}  d^{B']} ,\quad &{\rm if}\quad  j=k+1,\ldots,2n+1.
    \end{array}\right.
\end{equation*}We  take the nontrivial one with  $[A' B']=[0'1']$.

\begin{proof}[Proof of Theorem \ref{thm:k-invariant}]
It is direct to check that
\begin{equation*}
   \nabla_{A}^{ A' }\tau(\mathbf{q})_{ B'}^{B}=\delta_{A}^{ B}\delta_{B'}^{A'},
\end{equation*}
(cf. \cite[Lemma 3.1]{wan-wang}) for $ \nabla_{A}^{A' }$ given by \eqref{eq:nabla-x} and $\tau(\mathbf{q})_{ B'}^{B} $ given by \eqref{eq:q-z}. Therefore,
\begin{equation}\label{eq:nabla-nabla}
    \nabla_{A}^{ A' }\left[\underline F(\tau(\mathbf{q}))\right]=\frac {\partial \underline F}{\partial \mathbf{z}^{A}_{ A' }}(\tau(\mathbf{q}))
\end{equation}
for any holomorphic function $\underline F$ on $\mathbb{C}^{2n\times 2}$. By embedding \eqref{eq:q-z}, it is easy to see that
a complex valued polynomial $P(x_0, \cdots, x_{4n-1})$ on    $ \mathbb{ R}^{4n} $  can be extended naturally to a  holomorphic  polynomial
\begin{equation*}
   \underline P(\mathbf{z}):=P\left(\frac {\mathbf{z}_{0}^{ 0' }+\mathbf{z}_{1}^{ 1' }}2,\frac {\mathbf{z}_{0}^{ 0' }-\mathbf{z}_{1}^{ 1' }}{2\mathbf{i}},\frac {\mathbf{z}_{1}^{ 0' }-\mathbf{z}_{0}^{ 1' }}2,\frac {\mathbf{z}_{0}^{ 0' }+\mathbf{z}_{1}^{ 1' }}{-2\mathbf{i}}, \cdots\right)
\end{equation*}
  on $ \mathbb{
C}^{4n} $
satisfying
$
   P(\mathbf{q})=\underline P(\tau(\mathbf{q}))
$.
Then by \eqref{eq:nabla-nabla}, we have
\begin{equation*}
     d^{A'}  P(\mathbf{q})=\left (\underline d^
     {A'}  \underline P\right)  (\tau(\mathbf{q})).
\end{equation*}

If $j\leq k $,
  for a $ \mathcal  {V}_j   $-polynomial   $f= f_{\mathbf{A}}^{ \mathbf{A}'}(\mathbf{q}) s_{\mathbf{A}'} \omega^{\mathbf{A} }$ on $\mathbb{H}^n$,   we
  can construct a holomorphic $ \mathcal  {V}_j   $-polynomial  $\underline F= \underline F_{\mathbf{A} }^{\mathbf{A}'} (\mathbf{z})s_{\mathbf{A}'} \omega^{\mathbf{A}
  }$  on $ \mathbb{ C}^{2n\times 2}
    $   so that
$
   f(\mathbf{q})=\underline F(\tau(\mathbf{q})).
$
Then
\begin{equation*}
   \left[\pi_j(g)f\right] (\mathbf{q})= \left[\underline \pi_j(\tau(g))\underline F\right] (\tau(\mathbf{q}))
\end{equation*}for $g  \in{\rm
 {SL}}(n+1,\mathbb{H}) $, by comparing definition $\pi_j(g)$ in \eqref{eq:Uj<} with $\underline \pi_j(\tau(g))$ in \eqref{eq:action-U-j<}. Moreover, we have
 \begin{equation*}
     \partial^{[A'}  d^{B']} f(\mathbf{q})= \left(\partial^{[A'} \underline d^{B']} \underline   F \right)  (\tau(\mathbf{q})).
 \end{equation*}
So the invariance in
Theorem \ref{thm:k-invariant} follows from the identity \eqref{eq:DU-UD-X} in Theorem \ref{thm:DU-UD-X}. It is similar for the case $j\geq k$.\end{proof}

\begin{cor}$ \pi_j(g )$ in \eqref{eq:Uj<} and \eqref{eq:Uj>} satisfy the identity \eqref{eq:representation} of the  representations outside singularities.   \end{cor}

A domain $D$ is called {\it a domain of $k$-regularity} if ones cannot find two nonempty open sets $D_1$ and $D_2$ such that (1) $D_1$ is connected,
$D_1\nsubseteq D$ and $D_2\subset D_1\cap D$; (2) for each $f\in \mathcal{O}_k(D)$, there is a $\widetilde{ {f}}\in \mathcal{O}_k(D_1)$ satisfying $f=
\widetilde{{f}}$
on $D_2$.

\begin{cor}\label{cor:k-inv}  A linearly convex domain is a   domain  of $k$-regularity.
 \end{cor}
\begin{proof} Let $D$  be a linearly convex domain. Then for any $\mathbf{p}\in\partial D$,  there is a  hyperplane of quaternionic
dimension $n-1$ passing through $\mathbf{p}$ and not intersecting $D$. We can write the hyperplane as
\begin{equation*}
   \mathbf{a}+\mathbf{bq}=0
\end{equation*} for some $\mathbf{a}\in\mathbb{H}, \mathbf{b  }\in\mathbb{H}^n$. Then the
 $k$-regular function  \eqref{eq:k-regular-fraction} tends to infinity as $\mathbf{q}\rightarrow \mathbf{p}$, i.e. any boundary point is not $k$-regularly
 extendible. So $D$ is a domain of $k$-regularity.\end{proof}

By definition, it is easy to see that each convex domain in $ \mathbb{H}^n$ or the product $D=D_1\times\cdots\times D_n$ of
domains in $  \mathbb{H} $ is a linearly convex domain. In particular, any domain  in $  \mathbb{H} $ is linearly convex. Thus any domain  in $  \mathbb{H} $ is  a
domain  of $k$-regularity. But ones expect that the cohomologies of  the $k$-Cauchy-Fueter complex  on a domain $D\subset\mathbb{H}$  vanish if and only if the
domain $D$ is $k$-pseudoconvex   \cite{wang19}. This is different from the complex case, because the $\overline{\partial}$-complex on $\mathbb{C}$ is trivial,
while  the $k$-Cauchy-Fueter complex on $\mathbb{H}$  is not when $k>1$.

  The {\it quaternionic projective space} $\mathbb{H}P^{n }$ of dimension $n $ is the set of right quaternionic lines in $\mathbb{H}^{n+1}$. More
precisely,
$
\mathbb{H}P^{n }:=(\mathbb{H}^{n+1}\backslash\{0\})/\sim,
$
where $\sim$ is the equivalent relation: $\mathbf{ p }\sim \mathbf{q}$ in $\mathbb{H}^{n+1}$ if there is a non-zero quaternion
number $\lambda$ such that  $ \mathbf{p } =  \mathbf{q  }\lambda $. For a subset $E$ of $\mathbb{H}P^{n }$, the {\it dual complement}
$E^*$
is defined to be the set of hyperplane not intersecting $E$. For simplicity, assuming $E,E^* \subset \mathbb{H}^ {n }$ and $0\in E$, then the quaternionic version
of the {\it
Fantappi\`e transformation} is defined as
\begin{equation}\label{eq:Fantappie }
   \int_{E^*} \frac {(\mathbf{1}+\mathbf{b q})^{-1}   . s_{\mathbf{A}'} }{|\mathbf{1}+\mathbf{bq}|^{2 }} d\mu^{\mathbf{A}'}(\mathbf{b} ),
 \end{equation}
which  is  $k$-regular, where $\mu^{\mathbf{A}'}$'s are measures on $E^*$. It is an interesting question when any $k$-regular function on a set $E$ is the
 superposition of the simple fractions of the form \eqref{eq:k-regular-fraction}. In the complex case, it is known that the result holds if and only if $E$ is
 $\mathbb{C}$-convex
\cite[\S 3.6]{APS}.

  \subsection{Complexes over  locally   projective flat manifolds  } Let the parabolic subgroup ${\rm P}$ of $ {\rm
 {SL}}(n+1,\mathbb{H}) $ consist  matrices of the form
\begin{equation*}
   \left(  \begin{array}{cc}\mathbf   {a}& \mathbf  {b}\\
0& \mathbf  {d}
   \end{array}\right).
\end{equation*}Then the homogeneous space $ {\rm
 {SL}}(n+1,\mathbb{H})/ {\rm P}$ is the   quaternionic projective space  $\mathbb{H}P^{n }$ of dimension $n $.
\begin{equation*}
   \left(  \begin{array}{cc} \mathbf   {1} &0\\
\mathbf   {q}& \mathbf   {1}_{  n}
   \end{array}\right){\rm P} ,\qquad  \mathbf   {q}\in \mathbb{H} ^{n },
\end{equation*}
  constitute an open subset of $\mathbb{H}P^{n }$, which is diffeomorphic to quaternionic   space $\mathbb{H}^{ n }$.

Recall that  ${\rm Sp}(n ,1)$ is the group of all $(n+1)\times(n+1)$ quaternionic matrices which preserve the following hyperhermitian form:
\begin{align*}
Q(\mathbf{q},\mathbf{p})=- \overline{{q}_{1}}{p}_{1}-\cdots- \overline{{q}_{n}}{p}_{n}+
\overline{ {q}_{n+1 }}{p}_{n+1 },
\end{align*}
where $\mathbf{q}=(q_{1},\cdots,q_{n+1}),\ \mathbf{p}=(p_{1},\cdots,p_{n+1})\in \mathbb{H}^{n+1}$. It is a subgroup of $  \text{SL} (n+1,\mathbb{H})$.
Under the induced action of ${\rm Sp}(n ,1)$ on $\mathbb{H}P^{n }$,  $
D_{+}:=\{\mathbf{q} \in\mathbb{H}P^{n };Q(\mathbf{q} ,\mathbf{q} )>0\}
$ is an invariant subset which is equivalent to the \emph{quaternionic hyperbolic space} \cite{Shi-Wang}.
In this case we must have $q_{n+1}\neq0 $, and a point in  $D_{+}$ is equivalent to $
({q_{1}{q^{-1}_{n+1}}},\cdots,{q_{n }}{q^{-1}_{n+1}},1).$
 So we have the  ball model
for  quaternionic hyperbolic space: \begin{align*}
B^{4n }=\left\{\mathbf{q}\in\mathbb{H}^{n };| \mathbf{q}| <1\right\}.
\end{align*}
Thus the space $\mathcal{O}_k(B^{4n })$ of all $k$-regular functions on the ball $B^{4n } $
  is invariant under the action of the rank-$1$ Lie group ${\rm Sp}(n ,1)$.

A group $\Gamma$ is called \emph{discrete} if the topology on $\Gamma$ is the discrete topology.
We say that $\Gamma$ acts \emph{discontinuously} on a space $X$ at point $\mathbf{q}$ if there is a neighborhood $U$ of $\mathbf{q}$, such that $g(U)\cap
U=\emptyset$
for all
but finitely many $g$ of $ \Gamma$. Let $\Gamma$ be a discrete subgroup of $ \text{SL} (n+1,\mathbb{H}) $. Then   $\mathbb{H}P^{n }/\Gamma$
is a locally   projective flat manifold. In particular, if $\Gamma$ is a discrete subgroup of ${\rm Sp}(n ,1)\subset \text{SL} (n+1,\mathbb{H})$, then $  B^{4n
}/\Gamma$ is a locally   projective flat manifold. If $\Gamma$  is a cocompact or   convex cocompact
 subgroup of ${\rm Sp}(n ,1) $, then $  B^{4n }/\Gamma$ is a compact manifold without   or with  boundary, respectively.
In the latter case, the boundary is a
 spherical  quaternionic contact manifold (cf. \cite{Shi-Wang}).

Let $M$ be a locally   projective flat manifold with coordinates
charts $\{(U_{\alpha},\phi_{\alpha})\}$ with $\phi_{\alpha}:U_{\alpha}\rightarrow \mathbb{H}^n$, whose  transition
maps $
   \phi_{\beta}\circ \phi_{\alpha}^{-1}$ in \eqref{eq:charts}
 are   given by $ g_{\beta\alpha} \in \text{SL} (n+1,\mathbb{H})$ with the induced action \eqref{eq:g.q}:
   \begin{equation*}
      \xymatrix{
                &U_{\alpha}\cap U_{\beta} \ar[dr]^{\phi_{\beta}} \ar[dl]_{\phi_{\alpha}}            \\
\phi_{\alpha} (U_{\alpha}\cap U_{\beta}) \ar[rr]^{g_{\alpha\beta}^{-1}} & &    \phi_{\beta} (U_{\alpha}\cap U_{\beta})        }
   \end{equation*}It is obvious that
   \begin{equation*}
      g_{\alpha\beta}^{-1}=g_{\beta\alpha}.
   \end{equation*}
       $J_2 $ can be used to glue trivial bundles  $ \phi_{\alpha}(U_{\alpha})\times \mathbb{C}^{2n*}$  to obtain the bundle $E^*  $ by the transition functions of bundles given by
\begin{equation}\label{eq:equivalence-relation}
   \widehat{g}_{\beta\alpha} : \phi_{\alpha}(U_{\alpha}\cap U_{\beta})\times \mathbb{C}^{2n*}\rightarrow \phi_{\beta} (U_{\alpha}\cap U_{\beta})\times \mathbb{C}^{2n*},\qquad\left (\mathbf{q}, \omega^A\right)\mapsto  \left( g_{\alpha\beta}^{-1}.\mathbf{q},  J_2 \left(g_{\alpha\beta}^{-1}  , \mathbf{q} \right) .\omega^A\right).
\end{equation}
 Because the transition functions satisfy the compatibility condition:
 \begin{equation*} \begin{split} \widehat{g}_{ \gamma\beta}\circ\widehat{g}_{\beta\alpha}( \mathbf{q}, \omega^A )&= \widehat{g}_{ \gamma\beta}\left(  g ^{-1}_{\alpha\beta}.\mathbf{q},  J_2 \left(g ^{-1}_{\alpha\beta} , \mathbf{q} \right).\omega  ^ A\right) \\&=\left(  g_{\beta\gamma }^{-1}g ^{-1}_{\alpha\beta}.\mathbf{q}, \left[J_2
 \left(g_{\beta\gamma }^{-1}    ,g_{\alpha\beta}^{-1}. \mathbf{q} \right)J_2\left ( g_{\alpha\beta}^{-1}  ,\mathbf{q} \right)\right].\omega^A\right) \\
 &= \Big(   {g}_{\alpha \gamma}^{-1} .\mathbf{q},
 J_2
 \Big(g_{\alpha\gamma}^{-1}   ,  \mathbf{q} \Big)  .\omega^A\Big ),
  \end{split}   \end{equation*}
 by applying the cocycle identity \eqref{eq:pi-q}  in Proposition \ref{prop:cocycle} to $g_2^{-1}=g_{\beta\gamma }^{-1}, g_1^{-1}=  g_{\alpha\beta}^{-1}$.

 Similarly,  $J_1^{-1}$  can be used to glue trivial bundles  $ \phi_{\alpha}(U_{\alpha})\times \mathbb{C}^{2  }$  to obtain the bundle $H$
 by the transition functions of bundles given by
\begin{equation}\label{eq:equivalence-relation-2}
   \widehat{g}_{\beta\alpha} : \phi_{\alpha}(U_{\alpha}\cap U_{\beta})\times \mathbb{C}^{2 }\rightarrow \phi_{\beta} (U_{\alpha}\cap U_{\beta})\times \mathbb{C}^{2 },\qquad\left (\mathbf{q}, s_{ A'}\right)\mapsto  \left( g_{\alpha\beta}^{-1}.\mathbf{q},  J_1^{-1} \left(g_{\alpha\beta}^{-1}  ,\mathbf{q} \right) .s_{ A'}\right).
\end{equation}
    $J_1 $  can be used to  construct the bundle $ H ^*$.
$\wedge^\tau E $ is the $\tau$-th exterior product of the bundle $ E $, while $\odot^\sigma H^*$   is  $\sigma$-th symmetric product of the bundle  $ H^*$.
$\wedge^2H^*$ is a
complex line bundle defined similarly.
Moreover, we can define real line bundle
$\mathbb{R}[-1] $
by the transition functions of bundles given by
\begin{equation}\label{eq:equivalence-relation-1}
   \widehat{g}_{\beta\alpha} : \phi_{\alpha}(U_{\alpha}\cap U_{\beta})\times \mathbb{R} \rightarrow \phi_{\beta} (U_{\alpha}\cap U_{\beta})\times \mathbb{R} ,\qquad\left (\mathbf{q}, t\right)\mapsto  \left( g_{\alpha\beta}^{-1}.\mathbf{q}, \left|J_1^{-1} (g_{\alpha\beta}^{-1}  ,\mathbf{q} ) \right|^2 t\right).
\end{equation}
where
 \begin{equation*}\left|J_1^{-1} (g_{\alpha\beta}^{-1}  ,\mathbf{q} ) \right|^2 =
   \frac {1
   }{|\mathbf{a}+\mathbf{b}\mathbf{q} |^{2 }} ,\qquad {\rm for }\quad  g_{\alpha\beta}^{-1}=\left(  \begin{array}{cc}\mathbf   {a} &\mathbf   {b} \\
 \mathbf  {c} & \mathbf  {d}
   \end{array}\right)\in \text{SL} (n+1,\mathbb{H}).
\end{equation*}
 We can also define the bundles
$\mathbb{R}_{\pm}[-1]$ and  $\underline{\mathbb{R}}_{-} [-1]$ by  $  \mathbb{R} $ replaced by $\mathbb{R}_{+}=[0,\infty) $, $\mathbb{R}_{-}=( -\infty,0] $  and
$\underline{\mathbb{R}}_{-}=[ -\infty,0]  $, respectively. By definition, we have the isomorphism of 
complex line bundles:  
\begin{equation}\label{eq:line-bundles}
   \wedge^2H^*\cong \mathbb{C} [-1].
\end{equation}

When $j\leq k $, a global section of
$  \odot^{k-j }{H} \otimes \wedge^{j } {E}^* [-j-1] $ is
given by a family of local sections
   $f_\beta\in \Gamma
(\phi_{\alpha}(U_{\beta}),  \mathcal{ V}_{j })$ such that
\begin{equation}\label{eq:section-compatibility}
  \widehat{g}_{\beta\alpha}^*(f_\beta)=f_\alpha,
\end{equation} where
  $ \mathcal V_{j }= \odot^{k-j }\mathbb{C}^{2  }\otimes \wedge^j \mathbb{C}^{2n* }[-j-1]  $.
If writing
\begin{equation*}
   f_\beta(\mathbf{q} )=  (f_\beta)^{\mathbf{A}'}_ { \mathbf{A}}(\mathbf{q} ) s_ { \mathbf{A}'}\omega^{\mathbf{A} }
\end{equation*}as in \eqref{eq:f-supervariables}, substituting
  \eqref{eq:equivalence-relation}-\eqref{eq:equivalence-relation-1} into \eqref{eq:section-compatibility} and comparing it with the definition  of representation  $\pi_j(g)$ in
  \eqref{eq:Uj<},  we see that \eqref{eq:section-compatibility} can be  rewritten as
\begin{equation}\label{eq:equivalence-relation-pi}
\left[\pi_j(g_{\alpha\beta})  f_\beta\right](   \mathbf{q} ) =f_\alpha\left(   \mathbf{q} \right ).
\end{equation}
Thus   a family of local sections
   $f_\beta\in \Gamma
(\phi_{\beta}(U_{\beta}),  \mathcal{ V}_{j })$ give us a global section of
$  \odot^{k-j }{H}^* \otimes \wedge^{j } {E} [-j-1] $ if and only if \eqref{eq:equivalence-relation-pi} is satisfied.

It follows from   the ${\rm SL}(n+1,\mathbb{H})$-invariance of $\mathcal  {D}_j $  in Theorem \ref{thm:k-invariant} that
\begin{equation}\label{eq:D-bundle}
  \mathcal  {D}_j  f_\alpha (  \mathbf{q} )=\mathcal  {D}_j\Big[\pi_j(g_{\alpha\beta}) f_\beta\Big]( \mathbf{q} )=\Big[\pi_{j+1}(g_{\alpha\beta})
  \mathcal  {D}_j
  f_\beta\Big](  \mathbf{q} ) .
\end{equation}Namely,
$\{\mathcal  {D}_j f_\alpha\}$ gives us a section of
$  \odot^{k-j-1}{H}^* \otimes \wedge^{j+1} {E} [-j-2] $.
Therefore,
$\mathcal  {D}_j $ is a well defined operator between bundles:
\begin{equation*}
\mathcal  {D}_j :  \Gamma\left(M, \odot^{k-j }{H}\otimes \wedge^j {E}^*  [-j-1]\right)
\longrightarrow\Gamma \left(M, \odot^{k-j-1}{H} \otimes \wedge^{j+1} {E}^* [-j-2] \right).
\end{equation*}

It is similar for $j\geq k $. Thus we get    the $k$-Cauchy-Fueter complex  \eqref{eq:quaternionic-complex-diff} on locally  projective
flat manifolds.
 In particular, the $k$-th operator in the $k$-Cauchy-Fueter complex give us
 \begin{equation}
\mathcal  {D}_k=d^{  0' } d^{  1' }:  \Gamma\left( M, \wedge^k {E}^* [ -k-1] \right)
\longrightarrow\Gamma \left(  M,  \wedge^{k+2} {E}^* [ -k-2] \right).
\end{equation}
If $k=0$, it is
the Baston operator
 $
\triangle:  \Gamma\left( M, \mathbb{R} [ -1]\right)
\longrightarrow\Gamma \left(  M,  \wedge^2 {E}^* [ -2] \right).
$

\begin{rem} In \eqref{eq:D-bundle}, we apply  the ${\rm SL}(n+1,\mathbb{H})$-invariance   Theorem \ref{thm:k-invariant} to holomorphic functions  $f_\alpha$ locally defined on $\phi_\alpha(U_\alpha)$. But functions in the theorem are globally defined.  This can be done by approximating holomorphic functions  $f_\alpha$ on given convex domain   by polynomials.
\end{rem}

\section{The   quaternionic Monge-Amp\`{e}re operator on   locally   projective flat manifolds  }

  \subsection{The   cone   bundle $ \text{SP}^{2p}E^*$ of strongly positive $2p$-elements }

Recall that
a $ 2p $-form $\omega\in \wedge^{2p}\mathbb{C}^{2n *}$ is said to be \emph{elementary strongly positive} (cf. e.g. \cite[\S 3.1]{wan-wang} \cite{wang21}) if there
exist linearly
independent right $\mathbb{H}$-linear mappings $\eta_j:\mathbb{H}^n\rightarrow \mathbb{H}$ , $j=1,\ldots,p$, such that
\begin{equation}\label{eq:elementary-strongly} \omega=\eta_1^*\underline{\omega}^{0'}\wedge
\eta_1^*\underline{\omega}^{1'}\wedge\ldots\wedge\eta_p^*\underline{\omega}^{0'}\wedge \eta_p^*\underline{\omega}^{1'},\end{equation}
where
$\{\underline{\omega}^{0'},\underline{\omega}^{1'}\}$ is a basis of $\mathbb{C}^{2}$. The right $\mathbb{H}$-linear mapping $\eta_j$ is identified with a row vector in $\mathbb{H}^n$, and so
$\tau(\eta_j)$
is a
$2\times 2n $-complex matrix. Thus,
\begin{equation}\label{eq:elementary-strongly0}
   \eta_j^*\underline{\omega}^{A'}=\tau(\eta_j)^{A'}_ A \omega^A.
\end{equation}
In this section, we use the wedge product to denote the product of Grassmannian variables, which are consistent  with notations in pluripotential theory.

A $ 2p $-element   $\omega$  is called \emph{strongly positive} if it belongs to the convex cone
$\text{SP}^{2p}\mathbb{C}^{2n* }$  generated by elementary strongly positive $2p$-elements.
It is said to be \emph{positive} if for any   strongly positive element
$\eta\in \text{SP}^{2n-2p}\mathbb{C}^{2n* }$, $\omega\wedge\eta$ is positive.
The trivial cone   bundles $U_\alpha\times\text{SP}^{2p}\mathbb{C}^{2n* }$ can be glued by $ J_2   $ to a  cone   bundle $ \text{SP}^{2p}E^* $, a subbundle of $
\wedge^{2p}E^* $,  by the
following proposition.

\begin{prop} If $\omega $ is (elementary  strongly or strongly) positive $2p$-form, then  $ J_2(g^{-1}  ,\mathbf{q} ) . \omega$ is (elementary
strongly or strongly) positive $2p$-form for $
   g  \in {\rm
 {SL}}(n+1,\mathbb{H})$.
\end{prop}
\begin{proof} If $ \omega$ is a elementary strongly positive $2p$-form with $\omega$ given by \eqref{eq:elementary-strongly}, then  we have
\begin{equation*}\begin{split}
  J_2(g ^{-1} ,\mathbf{q} ) .\omega =&    J_2(g^{-1}  ,\mathbf{q} ) .\eta_1^*\underline{\omega}^{0'}\wedge
 J_2(g^{-1}  ,\mathbf{q} ) .\eta_1^*\underline{\omega}^{1'}\wedge\ldots\wedge J_2(g^{-1}  ,\mathbf{q} ) .\eta_k^*\underline{\omega}^{0'}\wedge {J
 }_2(g^{-1}  ,\mathbf{q} ). \eta_p^*\underline{\omega}^{1'}\\
 =&  \widehat{\eta}_1^*\underline{\omega}^{0'}\wedge
 \widehat{\eta}_1^*\underline{\omega}^{1'}\wedge\ldots\wedge \widehat{ \eta}_p^*\underline{\omega}^{0'}\wedge  \widehat{\eta}_p^*\underline{\omega}^{1'},
\end{split} \end{equation*}for $
   g^{-1}=\left(  \begin{array}{cc}\mathbf{a}&\mathbf{b}\\
\mathbf{c}&\mathbf{d}
   \end{array}\right)\in {\rm
 {SL}}(n+1,\mathbb{H})$, i.e. $J_2(g ^{-1} ,\mathbf{q} ) .\omega $ is also a elementary strongly positive $2p$-form,
where
\begin{equation*}
   \widehat{ \eta}_j={ \eta}_j \cdot (\mathbf{d} -(\mathbf{c}+\mathbf{d}\mathbf{q} ) (\mathbf{a}+\mathbf{b}\mathbf{q})^{-1}\mathbf{b})  :
\mathbb{H}^n\rightarrow \mathbb{H},\qquad  j=1,\ldots, p
\end{equation*}are linearly
independent right $\mathbb{H}$-linear mappings (row vectors),
since
\begin{equation*}\begin{split} J_2(g^{-1}  ,\mathbf{q} ) .\eta_j^*\underline{\omega}^{A'}& =\tau(\eta_j)^{A'}_ A  \tau\left(\mathbf{d} -(\mathbf{c}+\mathbf{d}\mathbf{q} ) (\mathbf{a}+\mathbf{b}\mathbf{q})^{-1}\mathbf{b}\right)^  A_{B}\omega^B
  =\widehat\eta_j^*\underline{\omega}^{A'}.
\end{split} \end{equation*}
Consequently, a $2p$-form   in the convex cone
$\text{SP}^{2p}\mathbb{C}^{2n* }$ is mapped by $
   g^{-1} \in {\rm
 {SL}}(n+1,\mathbb{H})$ to a form  also in this cone.  So (strongly)  positivity is preserved.
\end{proof}

  \subsection{ Closed positive  currents and ``integrals" }
\begin{prop}\label{prop:g-Vol} For $
    {g^{-1}}=\left(  \begin{array}{cc}\mathbf{a}&\mathbf{b}\\
\mathbf{c}&\mathbf{d}
   \end{array}\right)\in {\rm
 {SL}}(n+1,\mathbb{H})$,
   \begin{equation}\label{eq:g-Vol}
      T_{g^{-1}}^{*}  \operatorname{Vol}=\frac {1 }{|\mathbf{a}+\mathbf{bq}|^{4n+4}} \operatorname{Vol}
   \end{equation}where $ \operatorname{Vol}$ is the standard volume form of $\mathbb{R}^{4n}$.
\end{prop}
\begin{proof}
Recall that fractional linear transformation $\underline{T}_{g^{-1} }$ in (\ref{eq:g.z}) is a holomorphic mapping from $\mathbb{C}^{2n\times 2}$ minus a subspace
to $\mathbb{C}^{2n\times 2}$,
and
\begin{equation*}
  \operatorname{Vol}_\mathbb{C}:= \bigwedge_{A=0}^{2n-1} d\mathbf{z}^{A}_{0'}\wedge \bigwedge_{A=0}^{2n-1} d\mathbf{z}^{A}_{1'}
\end{equation*}
is a $(4n,0)$-form  on $\mathbb{C}^{2n\times 2}$. Its pull-back by $\underline{T}_{g^{-1} }$ is
 \begin{equation}\label{eq:g-Volc}
     \underline{T}_{g^{-1}}^{*} \operatorname{Vol}_\mathbb{C}=\frac {1 }{\det(\mathbf{a}+\mathbf{bz})^{2n+2}} \operatorname{Vol}_\mathbb{C}.
   \end{equation}If we denote  $X=\left.\frac d{dt}\underline{T}_{g_t^{-1}}^{*}\right|_{t=0}$, it is sufficient to prove
\begin{equation}\label{eq:X-vol}
   X. \operatorname{Vol}_\mathbb{C} =-(2n+2)\operatorname{tr}( \mathbf{\hat a}+\mathbf{\hat b}\mathbf{z} )  \operatorname{Vol}_\mathbb{C} ,\qquad {\rm for}\quad  X=\left(  \begin{array}{cc}\mathbf{\hat a}&\mathbf{\hat b}\\
\mathbf{\hat c}&\mathbf{\hat d}
   \end{array}\right)\in  \mathfrak{ sl}(2n+2,\mathbb{C}) .
\end{equation}

To show \eqref{eq:X-vol}, note that  for $ g^{-1} \in {\rm SL}(2n+2, \mathbb{C})$ given by \eqref{eq:g-1}, we have
 \begin{equation}\label{eq:g-dz}\begin{split}
   \underline{ T}_{g^{-1}}^{*} d\mathbf{z}_{ A'}^ {A} &=  {d} \left(\mathbf{  c}+\mathbf{\hat d}\mathbf{z} \right) \left(\mathbf{  a}+\mathbf{ b}\mathbf{z})^{-1}\right)_{
   A'}^{A}=   {d}
    \left[(\mathbf{  c}+\mathbf{  d}\mathbf{z} )_{E'}^{A} [(\mathbf{  a}+\mathbf{  b}\mathbf{z})^{-1}]_{A'}^{E'}\right] \\
      &=  \mathbf{  d}_{B}^ { A }d\mathbf{z}_{B'}^ {B} \left[(\mathbf{ a}+\mathbf{  b}\mathbf{z})^{-1}\right]_{A'}^ {B'}     - (\mathbf{ c}+\mathbf{ d}\mathbf{z}
      )_{E'}^ {A}\left[(\mathbf{  a}+\mathbf{  b}\mathbf{z})^{-1}\right]_{G'}^ {E'}\mathbf{  b}_{B}^ {G'}d\mathbf{z}_{B'}^
      {B}[(\mathbf{  a}+\mathbf{  b}\mathbf{z})^{-1}]_{A'}^ {B'}
      \\
            &=\left [\mathbf{  d}     - (\mathbf{ c}+\mathbf{  d}\mathbf{z} ) (\mathbf{  a}+\mathbf{  b}\mathbf{z})^{-1} \mathbf{  b}  \right]_{ B}^{ A } d\mathbf{z}_{B'}^{
            B}
           [ (\mathbf{  a}+\mathbf{  b}\mathbf{z})^{-1}]_{A'}^{B'}\\
           &=\left [ J_2 (g^{-1} ,\mathbf{z} )  \right]_{ B}^{ A } d\mathbf{z}_{B'}^{
            B}
           [  J_1 (g^{-1} ,\mathbf{z} )^{-1}]_{A'}^{B'}.
   \end{split} \end{equation}

 Consider the subgroup of one parameter: $g_t^{-1}=\exp(  tX)=I+tX+O(t^2)$ for $X=\left(  \begin{array}{cc}0& \mathbf{\hat b}\\
 0& 0
   \end{array}\right)\in  \mathfrak{ sl}(2n+2,\mathbb{C})$. Differentiate \eqref{eq:g-dz} for  $g_t^{-1}  $
  to get
\begin{equation}
  X.d \mathbf{z}_{ A'}^{A} =-\left (\mathbf z\mathbf{\hat b} \right)_{ B}^{ A} d\mathbf{z}_{A'}^{B}- d\mathbf{z}_{ B'}^{A}\left ( \mathbf{\hat b}\mathbf{z}\right)_{ A'}^{B'}   .
\end{equation}
Then it is direct to see that
\begin{equation*} \begin{split}
 X. \operatorname{Vol}_\mathbb{C} = &- \sum_{A=0}^{2n-1}\left\{ d\mathbf{z}_{0'}^{ 0}\wedge\cdots \wedge  \left[ \left(\mathbf z\mathbf{\hat b}\right )_{ B}^{ A } d\mathbf{z}_{ 0' }^{B}
 +
 d\mathbf{z}_{B'}^{A} \left( \mathbf{\hat b}\mathbf{z}\right)_{0'}^{B'} \right] \wedge\cdots \wedge
 d\mathbf{z}_{0'}^{ 2n-1 }\wedge\bigwedge_{A=0}^{2n-1} d\mathbf{z}_{ 1'}^{A}\right.\\
&-\bigwedge_{A=0}^{2n-1} d\mathbf{z}_{ 0'}^{A}\wedge\sum_{A=0}^{2n-1} d\mathbf{z}_{1'}^{0 }\wedge\cdots \wedge \left. \left[ \left(\mathbf z\mathbf{\hat b} \right)_{B}^{A}
d\mathbf{z}_{1'}^{B } +
 d\mathbf{z}_{B'}^{A }\left ( \mathbf{\hat b}\mathbf{z}\right)_{ 1'}^{B'} \right]
\wedge\cdots \wedge   d\mathbf{z}_{1'}^{ 2n-1 }\right \}\\
=& -\left(2\operatorname{tr}\left(\mathbf z\mathbf{\hat b}\right)  +2n\operatorname{tr}\left( \mathbf{\hat b}\mathbf{z}\right) \right)  \operatorname{Vol}_\mathbb{C}
\\
=& -(2n+2)\operatorname{tr}\left( \mathbf{\hat b}\mathbf{z}\right)   \operatorname{Vol}_\mathbb{C},
 \end{split} \end{equation*}by $\operatorname{tr}( \mathbf{\hat b}\mathbf{z})   = \mathbf{\hat b}_{A}^{ A'}\mathbf{z}_{A'}^{ A}  =\operatorname{tr}(\mathbf{z} \mathbf{\hat b})
 $.

    For $X= \left(  \begin{array}{cc} \mathbf{\hat a} & 0\\
 0& \mathbf{\hat d}
   \end{array}\right)$ with $\operatorname{tr}(  \mathbf {\hat a}) +\operatorname{tr}( \mathbf {\hat d}) =0$, differentiate \eqref{eq:g-dz} for $g_t^{-1}=\exp( tX)$
  to get
$
  X.d \mathbf{z}^{A}_{ A'} =  \mathbf{\hat d}  _{ B}^{ A} d\mathbf{z}^{B}_{ A'}-  d\mathbf{z}^{A}_{ B'}  \mathbf{\hat a} ^{B'}_{  A'}   .
$
Similarly, we get
\begin{equation*} \begin{split}
  X. \operatorname{Vol}_\mathbb{C} = &  (2\operatorname{tr}(\mathbf {\hat d})  -2n\operatorname{tr}( \mathbf {\hat a}) ) \operatorname{Vol}_{\mathbb{C}}
 = -(2n+2)\operatorname{tr}( \mathbf{\hat a} )  \operatorname{Vol}_\mathbb{C}.
 \end{split} \end{equation*} While  for $X= \left(  \begin{array}{cc} 0 & 0\\
\mathbf{\hat c}&0
   \end{array}\right)$,  $X.d \mathbf{z}_{A}^{A'}   =0$ and so $X. \operatorname{Vol}_\mathbb{C} =0$. The transformation formula \eqref{eq:g-Volc} is proved.

 When  pulled  back  to $\mathbb{R}^{4n }$  by  $\tau $,
 \begin{equation}\label{eq:q-dz} \tau^* \left( d\mathbf{z}_{A'}^{ A}
                           \right):=\left(
                                      \begin{array}{rr}
                                                                               \vdots\qquad &\vdots\qquad \\
                                         dx_{4l}+\textbf{i}dx_{4l+1} & -dx_{4l+2}-\textbf{i}dx_{4l+3} \\
                                            dx_{4l+2}-\textbf{i}dx_{4l+3} & d x_{4l }-\textbf{i}dx_{4l+1} \\
                                          \vdots\qquad &\vdots\qquad
                                      \end{array}
                                    \right).
\end{equation} 
  Thus
\begin{equation*}
  \tau^*\left(  d\mathbf{z}_{0'}^{ 2l }\wedge  d\mathbf{z}_{ 0'}^{2l+1}\wedge   d\mathbf{z}_{ 1'}^{2l} \wedge d\mathbf{z}_{ 1'}^{2l+1}\right)= 4 dx_{4l}\wedge
  dx_{4l+1}\wedge
  dx_{4l+2}\wedge
    dx_{4l+3} \end{equation*}
   and so
   $
    \tau^*   \operatorname{Vol}_{\mathbb{C}} = 4^n\operatorname{Vol},
$ where  $\operatorname{Vol}=dx_0\wedge\cdots
     dx_{4n-1}$. Since $T_{g^{-1}}=\tau^{-1}\circ\underline{T}_{\tau(g)^{-1}}\circ\tau$, we find that
     \begin{equation*} \begin{split}
    T_{g^{-1}}^{*}  \operatorname{Vol}&=   4 ^{-n} T_{g^{-1}}^{*}\tau^*  \operatorname{Vol}_{\mathbb{C}} =  (-4)^{-n} \tau^* \underline{T}_{\tau(g)^{-1}}^{*}
    \operatorname{Vol}_{\mathbb{C}} \\&
    = 4 ^{-n}\tau^* \left [\frac {1 }{\det(\tau(\mathbf{a})+\tau(\mathbf{b})\tau(\mathbf{q} ) )^{2n+2}} \operatorname{Vol}_\mathbb{C}\right] = \frac {1 }{| \mathbf{a}+\mathbf{bq}
    |^{4n+4}} \operatorname{Vol}
 \end{split} \end{equation*}by \eqref{eq:g-Volc}.
     The Proposition is proved.
\end{proof}

\begin{rem} \eqref{eq:g-dz} implies the
  well known  decomposition of
  the complexified cotangent  bundle   into a tensor product $
  \mathbb{C} T^*M\cong H \otimes E^*
 $ as $G_0$-modules.  It is Salamon's $EH$ formalism \cite{Sa82}.
\end{rem}

\begin{cor} \label{cor:E-Vol} On a   locally   projective flat manifold   $M $, $
     \wedge^{2n} {E}^* \cong \mathbb{R}[- 1]$ and $ \wedge^{4n} T^*M\cong\mathbb{R}[-2n-2].
 $
\end{cor}
\begin{proof}  It is direct to check that
\begin{equation*}
   \tau^* \left( d\mathbf{z}_{A'}^{ A}
                           \right)\left(\nabla_ B^{B'}\right)=\delta_{A}^{ B}\delta_{B'}^{A'}.
\end{equation*}
For $
   g^{-1}=\left(  \begin{array}{cc}\mathbf{a}&\mathbf{b}\\
\mathbf{c}&\mathbf{d}
   \end{array}\right)\in {\rm SL}(n+1,\mathbb{H})$,  $
 \tau(  g^{-1}) \in  {\rm SL}(2n+2,\mathbb{C}) $. So the identity \eqref{eq:triangular-decomposition} implies that
   \begin{equation}\label{eq:q-det}
 \det\left [\tau(\mathbf{d})   - \tau(\mathbf{c}+\mathbf{d}\mathbf{q})\tau (\mathbf{a}+\mathbf{b}\mathbf{q})^{-1}\tau( \mathbf{b})  \right]=   \det\left
 [\tau(\mathbf{a}+\mathbf{bq})   \right]^{-1}=  \frac {1 }{|\mathbf{a}+\mathbf{bq}|^{2}}
 .
\end{equation}The first  isomorphism holds by definition
\eqref{eq:equivalence-relation-1} of  $\mathbb{R}[- 1]$.
The second one follows from    Proposition \ref{prop:g-Vol}.
\end{proof}

On a   locally   projective flat manifold   $M $, denote by $\mathscr{D} (M,\wedge^{p}E^* [-p ])$
the space $
 C_0^\infty(M,\wedge^{p}E^* [-p  ])$,
    elements of which   are often called   {\it
$p$-forms}. An element $\eta\in \mathscr{D} (M, {\rm SP}^{2p}E^*\otimes \mathbb{R}_+[- l] )   $ is called a  {\it strongly positive $2p$-form}, while $\psi\in
\mathscr{D}(\Omega,\wedge^{p}E^*  [ -p ])$ is called \emph{closed} if
$
  \widehat{ \mathcal  {D}}  \psi=0
$ where
\begin{equation*}
   \widehat{ \mathcal  {D}} =s^{ [A'} d^{B']} : \Gamma \left(\wedge^{p} {E}^* [-p]\right) \rightarrow \Gamma \left({H}^*\otimes
\wedge^{p+1} {E}^* [-p-1]\right)
\end{equation*}
  given by \eqref{eq:widetilde-D} is an invariant operator, which  is the $(p-1)$-th operator in the $(p-2)$-Cauchy-Fueter
complex (here we assume $p\geq 2$ for simplicity). It is equivalent to
\begin{equation*}
   d^{0'}\psi =  d^{1'} \psi=0
\end{equation*}
locally. Note that $d^{A'}$ is not  an invariant operator, but $s^{ [A'} d^{B']} $ is.

As in the flat case, we can  define ``integral" for a
  $2n$-form.  Assuming $M$ is orientable,  there exists a global nonvanishing section of $\wedge^{4n} T^*M$, the volume form, say $dV$. By Corollary \ref{cor:E-Vol}, we have
  \begin{equation*}
     \wedge^{2n} {E}^* [ -2n-1] \cong \mathbb{R}[ -2n-2] \cong  \wedge^{4n} T^*M.
  \end{equation*}
   Thus for a section $\omega\in\mathscr{D} (M,  \wedge^{2n} {E}^* [ -2n-1]  )$,
  there exists a function $f$ on $M$ such that
$
     \omega\cong f dV,
$
  and so the functional on $\mathscr{D} (M, \wedge^{2n} {E}^* [ -2n-1])$ defined by
  \begin{equation}\label{eq:integral}\int_M \omega:=\int_M f dV,
\end{equation}
is well defined.
\begin{cor} For $u_0, \cdots, u_n\in \Gamma(M, \mathbb{R}[-1])$,
$
   u_0\triangle u_1\wedge\cdots\wedge\triangle u_n\in \Gamma(M, \wedge^{2n}
{E}^* [- 2n-1])
$
and
\begin{equation*}
 \int_K  u_0\triangle u_1\wedge\cdots\wedge\triangle u_n
\end{equation*}
for a compact subset $K$ is well defined.
\end{cor}
\begin{proof}  Since $\triangle u_j\in \Gamma(M, \wedge^{2 }
{E}^* [- 2 ])$, we have $\triangle u_1\wedge\cdots\wedge\triangle u_n\in \Gamma(M, \wedge^{2n}
{E}^* [- 2n ])$ by definition.
\end{proof}

 An element of the space
$[\mathscr{D}  (M,\wedge^{2n-p}E^* [-(2n-p)-1])]'$ dual  to $\mathscr{D}  (M,\wedge^{2n-p}E^* [-(2n-p)-1])$  is called a {\it $p$-current}.
$\psi\in \mathscr{D}  (M,\wedge^{p}E^*  [ -p ])$ defines a $p$-current by
\begin{equation}\label{eq: T-psi}
  T_\psi (\eta)=\int_M\psi \wedge \eta
\end{equation}for any $\eta\in\mathscr{D}  (M,\wedge^{2n- p}E^* [-(2n-p)-1])$,
since $\psi \wedge \eta \in \mathscr{D} (M, \wedge^{2n} {E}^* [ -2n-1])$.  A $p$-current $T$ is called \emph{closed} if
\begin{equation*}
     ( \widehat{  \mathcal  {D}} T)(\eta):= T( \mathcal  {D} \eta)=0
\end{equation*}for any  $\eta\in \mathscr{D}  (M,H \otimes\wedge^{2n-p-1 }E^* [-(2n-p) ])$, where $\mathcal  {D} \eta\in \mathscr{D}  (M, \wedge^{2n-p
}E ^*[-(2n-p)-1 ])$ by definition. It is direct to check that
$
    \widehat{  \mathcal  {D}} T_\psi= (-1)^{p-1}T_{ \widehat{  \mathcal  {D}}\psi}
$
with the natural extension of \eqref{eq: T-psi} to the dual pair between $  H \otimes\wedge^{2n-p -1}E^* [-(2n-p)] $ and $  H^*\otimes\wedge^{ p+1 }E^* [-  p-1 ] $. We
omit details.

 A $2p$-current $T$ is said to be \emph{positive} if we have
$T(\eta)\geq0$ for any strongly positive form $\eta\in \mathscr{D}  (M, \operatorname{SP}^{2n-2p}E^*\otimes \mathbb{R}_+ [-(2n-2p)-1])$.
A    upper
semicontinuous  section of $  {{\mathbb{R}}}[-1]$
   is said to be \emph{plurisubharmonic}  if    $\triangle u$ is
a closed positive $2$-current. The space of
plurisubharmonic  section on $M$ is denoted by $ {\rm PSH} (M)$. For $u\in {\rm PSH} (M)\cap C^2(M,  \mathbb{R}  [-1])$,  $\triangle u$ is a closed   strongly
positive
$2$-form.

Now we fix a metric $h$ for the line bundle $\mathbb{R}[-1]$.
 \begin{thm}\label{thm:estimate-C0} On a  locally   projective flat manifold $M$,
for any  $u_0,\ldots u_n\in {\rm PSH} (M)\cap C^2(M,  \mathbb{R}_- [-1])$, we have
\begin{equation}\label{3.14}0\leq \int_M - u_0\triangle u_1\wedge\cdots\wedge\triangle u_n \leq
C\prod_{i=0}^n\|u_i\|_{L^\infty},
\end{equation}where $L^\infty$-norms are defined in terms of the metric $h$. \end{thm}
Its proof is reduced to known Chern-Levine-Nirenberg  estimate \cite{alesker1,wan-wang}  on flat quaternionic  space $\mathbb{H}^n$ by the unit partition. As a
corollary, for $u_j\in  {\rm
PSH} (M)\cap C(M,  \mathbb{R}_- [-1])$,
\begin{equation*}
 -  u_0\triangle u_1\wedge\cdots\wedge\triangle u_n
\end{equation*}defines a measure on $M$.
The capacity of a compact subset $K$ of $M$ can be defined as
\begin{equation*}
  {\rm cap} (K,M):=\sup\left\{-\int_K  u_0\triangle u_1\wedge\cdots\wedge\triangle u_n;u_j\in {\rm PSH} (M)\cap C(M,  \mathbb{R}_- [-1]) ,  |u_j|_h\leq 1\right
  \}.
\end{equation*}The
 Monge-Amp\`{e}re equation and
pluripotential theory  on  a  locally   projective flat manifold will be discussed in the subsequent part.
 \section{Projective  invariance    }
 Write
$
   (\triangle u)^n=M(u)\Omega_{2n}
$
locally,  where
$ \Omega_{2n}:=\omega^0\wedge
\omega^1\wedge\ldots\wedge\omega^{2n-2}\wedge
\omega^{2n-1}.$

\begin{prop}\label{prop:projectively-invariant-operator}
  $
     \frac { M( u)}{u^{2n+1} }
$
is   quaternionic projectively invariant   for $u\in \Gamma(M, \mathbb{R}[-1])$.
\end{prop}
\begin{proof}
Note that for  $u\in \Gamma( M, \mathbb{R}[-1])$, we have  locally
 \begin{equation*}
  \triangle\left( \frac {1 }{|\mathbf{a}+\mathbf{bq}|^{2}} u(g^{-1}.\mathbf{q})\right)= \frac {1 }{|\mathbf{a}+\mathbf{bq}|^{4 }} J_2(g^{-1},\mathbf{q}) . \triangle
  u  (g^{-1}.\mathbf{q}) ,
\end{equation*} by the ${\rm SL}(n+1,\mathbb{H})$-invariance of $\mathcal  {D}_0 $ in the $0$-Cauchy-Fueter complex  in Theorem \ref{thm:k-invariant}. Thus,
\begin{equation}\label{eq:triangle-transform}
  \left(\triangle\left( \frac {1 }{|\mathbf{a}+\mathbf{bq}|^{2}} u(g^{-1}.\mathbf{q})\right)\right)^n= \frac {1 }{|\mathbf{a}+\mathbf{bq}|^{4n }}
  J_2(g^{-1},\mathbf{q}) . \left(\triangle  u  (g^{-1}.\mathbf{q})\right)^n .
\end{equation}
Consequently,
  we get
\begin{equation*}\begin{split}
 M\left( \frac {1 }{|\mathbf{a}+\mathbf{bq}|^{2}} u(g^{-1}.\mathbf{q})\right)\Omega_{2n}& = \frac {1 }{|\mathbf{a}+\mathbf{bq}|^{4n }} M( u) (g^{-1}.\mathbf{q})
 J_2(g^{-1},\mathbf{q}) .\Omega_{2n}\\&= \frac {1 }{|\mathbf{a}+\mathbf{bq}|^{4n+2 }} M( u) (g^{-1}.\mathbf{q})  \Omega_{2n},
\end{split} \end{equation*}by using \eqref{eq:q-det},
and so
\begin{equation*}
   M\left( \frac {1 }{|\mathbf{a}+\mathbf{bq}|^{2}} u(g^{-1}.\mathbf{q})\right)= \frac {1 }{|\mathbf{a}+\mathbf{bq}|^{4n+2}} M( u) (g^{-1}.\mathbf{q}) .
\end{equation*}
It implies $
     \frac { M( u)}{u^{2n+1} }
$
is a quaternionic projectively invariant   for $u\in \Gamma( M, \mathbb{R}[-1])$.\end{proof}

Now consider the   problem:
\begin{equation}\label{eq:M=1}\left\{\begin{aligned}&
     \frac { M( u)}{u^{2n+1} }=1,\\&
     u|_{b D}=\infty.
\end{aligned}\right.\end{equation} for $u\in C( {D},  \mathbb{R}_- [-1])$, where $D$ is a domain in $M$.
Letting $u=-\frac 1\varrho $, then $\varrho$ is a section of $\mathbb{ R}_+[1]$. Note that locally $\triangle=d^{0'} d^{1'}$  and
\begin{equation*}
   \triangle  \left (-\frac 1\varrho \right)= \frac {\triangle \varrho}{\varrho^2}-\frac {2d^{0'}\varrho\wedge d^{1'}\varrho}{\varrho^3}.
\end{equation*}
Then,
\begin{equation}\label{eq:J0}
  \left (\triangle   \left(-\frac 1\varrho\right)\right)^n =\frac {1}{\varrho^{2n+1}}\left[ \varrho(\triangle \varrho)^n-2n d^{0'}\varrho\wedge
  d^{1'}\varrho\wedge (\triangle \varrho)^{n-1}\right].
\end{equation}
If we define $J(\varrho)$ locally by
\begin{equation}\label{eq:J}
 -  J(\varrho)\Omega_{2n}:=  \varrho(\triangle \varrho)^n-2n d^{0'}\varrho\wedge d^{1'}\varrho\wedge (\triangle \varrho)^{n-1} ,
\end{equation}then
\eqref{eq:J0} implies
\begin{equation*}
    J(\varrho)\Omega_{2n}= \frac { M( u)}{u^{2n+1} }\Omega_{2n}.
\end{equation*}By Proposition \ref{prop:projectively-invariant-operator},
  $
       J(\varrho)
$
is   quaternionic projectively invariant, i.e. it is a well defined scalar function on $M$.
Therefore, the   problem
\eqref{eq:M=1} is equivalent to the Dirichlet problem
\begin{equation}\label{eq:Dirichlet}\left\{\begin{aligned}&
     J(\varrho) = 1,\\&
     \varrho|_{b D}=0.
\end{aligned}\right.\end{equation}

Fefferman \cite{Feff}  used the complex Monge-Amp\`{e}re operator to construct a holomorphically invariant defining density of a strictly pesudoconvex domain, and
CR invariant
differential operators  on the boundary.  We  construct a projectively invariant  defining density.
\begin{thm} \label{thm:Fefferman-density} For a domain $D$ in  a locally   projective flat manifold $M$, there exists a defining density $\varrho$, a section of $
\mathbb{ R}_+[1]$,  such that
\begin{equation*}
   J( \varrho)= 1+O(\varrho^{2n+2}).
\end{equation*}
Any smooth  local  approximate solution $\varrho\in C^\infty(\overline{D}, \mathbb{ R}_+[1])$ to this equation is
uniquely determined up to order $ 2n+2$.
\end{thm}
\begin{proof} It is sufficient to show the result locally.
For  a defining function $\varphi$ of the domain $D=\{\varphi>0\}$ and ${\rm grad} \varphi\neq 0$ on $\partial D$. We can assume $J(\varphi)=1$ on $\partial D$.
This is because for any smooth function $\eta$,
\begin{equation*}
   J(\eta\varphi)|_{b D}=\eta^{n+1}J( \varphi)|_{b D},
\end{equation*}we can choose $\eta=J( \varphi)^{\frac 1{n+1}}.$

Now suppose that for $s\geq 2$, we have
$
   J( \varphi)=1+O(\varphi^{s-1}).
$  We want to solve this equation for $s$ replaced by $s+1$, i.e.
\begin{equation}\label{eq:Jeq}
   J( \varrho)= 1+O(\varrho^{s }).
\end{equation}

Take $\varrho=\varphi+\eta\varphi^s$. Then
\begin{equation} \label{eq:J-1}\begin{split}
 -J(\varrho)\Omega_{2n} &=\varphi (\triangle \varrho)^n-2n(1+s\eta\varphi^{s-1})^2 d^{0'}\varphi\wedge  d^{1'}\varphi\wedge (\triangle \varrho)^{n-1}+O(\varphi^s),
 \end{split} \end{equation}
where
 \begin{equation}\label{eq:J-2}
    \triangle \varrho=(1+s\eta\varphi^{s-1})\triangle \varphi+s(s-1)\eta\varphi^{s-2} d^{0'}\varphi\wedge  d^{1'}\varphi+s \varphi^{s-1}  (d^{0'} \eta\wedge
    d^{1'}\varphi+d^{0'} \varphi    \wedge  d^{1'}\eta),
 \end{equation}
and so
 \begin{equation}\label{eq:J-3}
    d^{0'}\varphi\wedge  d^{1'}\varphi\wedge (\triangle \varrho)^{n-1}=(1+s\eta\varphi^{s-1})^{n-1}d^{0'}\varphi\wedge  d^{1'}\varphi \wedge (\triangle
    \varphi)^{n-1} ,
 \end{equation}
 by $d^{A'}\varphi\wedge d^{A'}\varphi=0$.
 By substituting \eqref{eq:J-2}-\eqref{eq:J-3} into \eqref{eq:J-1} and absorbing terms having factor $ \varphi^s $ into $O(\varphi^s)$, we find that
 \begin{equation*}\label{eq:Ju}\begin{split}
- J(& \varrho) \Omega_{2n}=  (1+s\eta\varphi^{s-1})^{n+ 1} \varphi \left [\triangle \varphi+s(s-1)(1+s\eta\varphi^{s-1})^{-1}\eta\varphi^{s-2} d^{0'}\varphi\wedge
d^{1'}\varphi\right]^n\\
&\qquad\qquad-  2n(1+s\eta\varphi^{s-1})^{n+1}d^{0'}\varphi\wedge  d^{1'}\varphi \wedge (\triangle \varphi)^{n-1}+O(\varphi^s)\\&=
(1+s\eta\varphi^{s-1})^{n+1}\left [ - J(\varphi )\Omega_{2n}+ns(s-1)(1+s\eta\varphi^{s-1})^{-1}\eta\varphi^{s-1} d^{0'}\varphi\wedge  d^{1'}\varphi\wedge
(\triangle
\varphi)^{n-1}  \right]+O(\varphi^s)\\&=
(1+s\eta\varphi^{s-1})^{n+1}\left [- J(\varphi )\Omega_{2n}+ \frac 12s(s-1)(1+s\eta\varphi^{s-1})^{-1}\eta\varphi^{s-1} J(\varphi )\Omega_{2n}\right]+O(\varphi^s)
\\&=
\frac {-J(\varphi )\Omega_{2n}}{1-[s(n+1)- s(s-1)/2] \eta\varphi^{s-1}} +O(\varphi^s),
 \end{split} \end{equation*}
where we have used
 \begin{equation*}
    (1+s\eta\varphi^{s-1})  \varphi=\varphi+O(\varphi^s),\qquad (\varphi^{s-1})^2=O(\varphi^s).
 \end{equation*}

 Thus $ J( \varrho)= 1+O(\varphi^s)$ is equivalent to
 \begin{equation*}
    \frac s2(2n+3- s )\eta\varphi^{s-1}=1-J(\varphi )+O(\varphi^s)
 \end{equation*}
  Namely, if we   take
 \begin{equation*}
   \varrho=\varphi\left(1+2\frac {1 -J(\varphi )}{s(2n+3- s )}\right),
 \end{equation*}
 \eqref{eq:Jeq} is solvable for $s=2,\ldots,2n+2 $.
\end{proof}

We call the  defining density given by Theorem  \ref{thm:Fefferman-density} a {\it Fefferman defining density}. If  $u=-\frac 1\varrho  $ is a PSH section of
$\mathbb{ R}_-[-1]$,
it defines a projectively
invariant positive $2$-form
\begin{equation*}
   \triangle u= \frac {\triangle \varrho}{\varrho^2}-\frac {2d^{0'}\varrho\wedge d^{1'}\varrho}{\varrho^3},
\end{equation*}
on $D$.
This can be viewed as the  quaternionic  version of the Blaschke metric on strictly convex domains \cite{Marugame18,Sasaki85}. Discussion of this form and
applications to  quaternionic strictly pseudoconvex boundary will appear in   the subsequent part.

\end{document}